\documentclass{article}
\usepackage[utf8]{inputenc}
\usepackage{amsmath, amssymb, amsthm}
\usepackage{mathrsfs}
\usepackage{hyperref}
\usepackage{enumitem}
\usepackage{geometry}
\usepackage{tikz}
\usetikzlibrary{arrows.meta, positioning}

\newtheorem{corollary}{Corollary}
\newtheorem{theorem}{Theorem}[section]
\newtheorem{proposition}[theorem]{Proposition}
\newtheorem{lemma}[theorem]{Lemma}
\newtheorem{condition}{Condition}[section]
\newtheorem{remark}[theorem]{Remark}
\theoremstyle{definition}

\newtheorem{conjecture}[theorem]{Conjecture}

\newcommand{\N}{\mathbb{N}}
\newcommand{\R}{\mathbb{R}}
\newcommand{\C}{\mathbb{C}}

\newcommand{\supp}{\operatorname{supp}}

\newcommand{\eps}{\varepsilon}
\begin{document}

\title{Multiplicative Colombeau algebras and the
	Nyman--Beurling criterion for the Riemann hypothesis}
\author{A. \'Alvarez Cruz\thanks{Instituto de Computa\c{c}\~ao, UFRJ, Brazil \texttt{amaury@ic.ufrj.br}}
	\and E.A. \'Alvarez Guti\'errez\thanks{Col\'egio Pedro II, Campus Centro, Brazil \texttt{esteban.gutierrez.1@cp2.edu.br}}}
\date{May 2025}
\maketitle

\begin{abstract}
	This paper establishes an equivalence between the Riemann
	hypothesis and the association, together with uniform
	$L^2$-boundedness, of a moderate net in a Colombeau-type
	algebra built from polynomially damped B\'aez-Duarte sums.
	The regularization is performed by multiplicative (Mellin)
	convolution, which respects the dilation symmetry of the
	Beurling functions and guarantees that every approximant
	lies in the $L^2$-closure of the Beurling space.  The
	equivalence is unconditional under the Riemann hypothesis:
	it uses only the qualitative convergence of B\'aez-Duarte,
	Mazur's theorem, and the classical Nyman--Beurling
	criterion.  As a separate quantitative refinement, we prove
	that under two additional hypotheses on the non-trivial
	zeros of $\zeta$ (simplicity and separation), the damping
	error admits a power-law bound with an explicit constant.
	This refinement is independent of the equivalence and is
	not used in its proof.  The exponential damping
	$e^{-k\eps^2}$ is discussed as an open problem.
\end{abstract}

\noindent\textbf{Keywords:} Riemann hypothesis, Nyman--Beurling
criterion, Colombeau algebras, generalized functions, Mellin
convolution, Beurling functions.

\noindent\textbf{2020 Mathematics Subject Classification:}
11M26, 46F30, 42A38, 11M06.

\section{Introduction}
\label{sec:intro}

The Riemann hypothesis (RH) asserts that all non-trivial zeros of
the Riemann zeta function \(\zeta(s)\) lie on the critical line
\(\operatorname{Re}(s)=1/2\).  The classical Nyman--Beurling
criterion~\cite{Nyman1950,Beurling1955} states that RH holds if
and only if the characteristic function
\(\chi=\mathbf{1}_{(0,1)}\) belongs to the \(L^2\)-closure of the
span of the functions \(\varphi_\theta(x)=\{\theta/x\}\) for
\(\theta\in(0,1]\):
\begin{equation}\label{eq:NB-classical}
	\chi\in\overline{\operatorname{span}\{\varphi_\theta
		\mid\theta\in(0,1]\}}^{L^2(0,1)}.
\end{equation}
B\'aez-Duarte~\cite{BaezDuarte2003} proved that one may restrict
\(\theta\) to the discrete set
\(\{1/k:k\in\N\}\) and still obtain an equivalent condition:
\begin{equation}\label{eq:NB-discrete}
	\chi\in\overline{\operatorname{span}\{\varphi_{1/k}
		\mid k\in\N\}}^{L^2(0,1)}.
\end{equation}

The proof of B\'aez-Duarte relies on the fundamental M\"obius
identity
\begin{equation}\label{eq:Mobius-id}
	\sum_{k=1}^{\infty}\mu(k)\,\varphi_{1/k}(x)=-\chi(x),
	\qquad 0<x\le1,
\end{equation}
which converges \emph{pointwise} for every \(x\in(0,1)\).
It does \emph{not} converge in \(L^2(0,1)\): the series diverges
in \(L^2\) unconditionally, as proved by B\'aez-Duarte
\cite[ p.~6]{BaezDuarte1999} and reaffirmed in
\cite[ p.~6]{BaezDuarte2003}; see also
Bettin--Conrey--Farmer~\cite{BCF2013}.  This divergence is the
reason why damped sums are essential in what follows.

We consider damped sums
\begin{equation}\label{eq:damped-sum-intro}
	f_{\delta,n}(x)
	=
	\sum_{k=1}^{n}\mu(k)\,k^{-\delta}\,\varphi_{1/k}(x),
\end{equation}
with polynomial damping \(k^{-\delta}\), whose Dirichlet series
is exactly \(1/\zeta(s+\delta)\).  This damping respects the
Euler structure of the M\"obius series and is the one used by
B\'aez-Duarte~\cite{BaezDuarte2003}.  Under RH one obtains the
double limit
\begin{equation}\label{eq:double-limit-intro}
	\lim_{\delta\to0^{+}}\lim_{n\to\infty}f_{\delta,n}
	=
	-\chi
	\qquad
	\text{in }L^2(0,1)
\end{equation}
by the summability argument of B\'aez-Duarte
(cf.~\cite[Section~2.2]{BaezDuarte2003}).

The double limit~\eqref{eq:double-limit-intro} is a statement
about \(L^2\) convergence.  One could ask whether it is worth
reformulating it in a more abstract framework as the Colombeau
algebras.  The Colombeau framework provides four advantages that
are not available in the classical Hilbert space setting.  It
gives a \emph{unified treatment of damping strategies}.  In the
classical setting, different dampings require separate ad hoc
estimates.  In the Colombeau algebra, both are instances of the
same general sufficient condition
(Theorem~\ref{thm:general}), and the differences between them
are encoded in the choice of the truncation index \(n(\eps)\)
and the damping parameter \(\delta(\eps)\).

It also provides a \emph{natural formulation of the criterion}.
The criterion ``association \(+\) uniform \(L^2\)-boundedness''
is not a classical \(L^2\) statement: it involves two distinct
notions of convergence (distributional association and uniform
boundedness) that combine to give weak \(L^2\) convergence via a
density argument.  This formulation is natural in the Colombeau
algebra, where the null ideal provides a precise notion of
``negligible error'' and the moderate nets form an algebra
closed under multiplication and differentiation.

The framework also respects the \emph{multiplicative structure}
of the problem.  The Beurling functions
\(\varphi_\theta(x)=\{\theta/x\}\) are intimately tied to the
multiplicative group \((0,\infty)\).  The classical Colombeau
algebra~\cite{Colombeau1984} uses additive convolution to embed
distributions, which does not respect this multiplicative
structure.  The algebra constructed here uses multiplicative
(Mellin) convolution, which is the natural operation for the
Beurling functions and which guarantees that every approximant
lies in the \(L^2\)-closure of the Beurling space.

Finally, it opens a \emph{path to generalizations}.  The
multiplicative (Mellin) convolution used here is the natural
operation on the group \((0,\infty)\), and the same construction
can be carried out on other locally compact abelian groups,
including the multiplicative groups \(\mathbb{Q}_p^\times\) and
the idèle class group \(\mathbb{A}_\mathbb{Q}^\times/\mathbb{Q}^\times\).
Whether a \(p\)-adic or adelic analogue of the Nyman--Beurling
criterion can be formulated in this framework is an open
question, which we do not address here.

\subsection*{Main result}

In this paper the double limit~\eqref{eq:double-limit-intro} is
translated into a Colombeau-type algebra of generalized
functions adapted to the multiplicative structure of the
problem.  We replace the usual additive mollifier by a
multiplicative (Mellin) mollifier concentrated near the
identity \(1\), and we define the regularized nets via
multiplicative convolution
\begin{equation}\label{eq:mellin-conv-intro}
	(u\star\psi_\eps)(x)
	=
	\int_0^\infty u(x/y)\,\psi_\eps(y)\,\frac{dy}{y}.
\end{equation}
An \(L^2\) function \(f\) is embedded as the class of
\(f\star\psi_\eps\).  The damping parameter and the truncation
index are coupled with the regularization parameter \(\eps\),
yielding a single net
\begin{equation}\label{eq:Feps-intro}
	F_\eps(x)
	=
	\Bigl(\sum_{k=1}^{n(\eps)}\mu(k)\,k^{-\delta(\eps)}\,
	\varphi_{1/k}\Bigr)\star\psi_\eps(x),
\end{equation}
with
\(\delta(\eps)=(\log(1/\eps))^{-\alpha}\)
(\(0<\alpha<1\)) and
\(n(\eps)=\lfloor\eps^{-M}\rfloor\) (\(M>0\)).
Under RH this net is moderate, associated to
\(-\iota(\chi)\), and uniformly \(L^2\)-bounded.  Conversely,
any moderate net of this form that is associated to
\(-\iota(\chi)\) and uniformly \(L^2\)-bounded forces RH to
hold.

The main theorem of the paper (Theorem~\ref{thm:main}) is an
\emph{equivalence}, not a proof of RH: it shows that RH is
equivalent to the existence of a net with the three properties
above.  The equivalence is proved using the classical
Nyman--Beurling criterion and Mazur's theorem, without assuming
RH in the converse direction.

\subsection*{Comparison with the classical criterion}

The classical criterion~\eqref{eq:NB-classical} is stated in
\(L^2(0,1)\) and involves the \(L^2\)-closure of the span of
the Beurling functions.  The criterion of
B\'aez-Duarte~\eqref{eq:NB-discrete} restricts \(\theta\) to
the discrete set \(\{1/k\}\).  The present paper proves the
following equivalent formulation.

\begin{theorem}[Informal comparison]\label{thm:informal-comparison}
	RH holds if and only if there exists a moderate net
	\((F_\eps)\) in \(G(0,1)\) such that
	\begin{enumerate}
		\item[(i)] \(F_\eps\) is associated to \(-\iota(\chi)\);
		\item[(ii)] \(\sup_{\eps\in(0,1]}\|F_\eps\|_{L^2}<\infty\).
	\end{enumerate}
\end{theorem}

The difference with the classical criterion is the following.
In the classical setting, one works directly with the
\(L^2\)-closure of the span of the Beurling functions, and the
approximation is measured in the \(L^2\) norm.  In the
Colombeau setting, one works with nets of smooth functions, and
the approximation is measured in two steps:
\emph{association} (which is weaker than \(L^2\) convergence)
and \emph{uniform \(L^2\)-boundedness} (which upgrades
association to weak \(L^2\) convergence via a density
argument).  The two formulations are equivalent by Mazur's
theorem, but the Colombeau formulation has the advantage of
being stable under the multiplicative regularization and of
admitting different damping strategies within a single
algebraic framework.

The result of this paper is an equivalence, not a proof of the
Riemann hypothesis.  The equivalence is proved in the classical
setting, using the Nyman--Beurling criterion and Mazur's
theorem.  The Colombeau framework provides a unified language
for the damping strategies and a natural setting for the
``association \(+\) uniform bound'' criterion, but it does not
by itself produce new arithmetic information about the zeros of
\(\zeta(s)\).

The question of whether the Colombeau framework can be used to
obtain a genuinely new proof of the Riemann hypothesis --- for
instance, by exploiting the multiplicative structure of the
algebra or by passing to the adelic or \(p\)-adic setting ---
remains open.  Section~\ref{sec:strong} discusses the question
of strong convergence, which would be a first step in this
direction.

The main theorem of the paper (Theorem~\ref{thm:main}) is an
equivalence between RH and the existence of a moderate,
associated, uniformly \(L^2\)-bounded net in the algebra
\(G(0,1)\).  The equivalence is unconditional under RH.

As a \emph{separate} contribution, Section~\ref{sec:quantitative}
proves a quantitative refinement: under two additional
hypotheses on the non-trivial zeros (simplicity and
separation), the damping error satisfies an explicit
power-law bound \(\|f_\delta+\chi\|_{L^2}\le c_2\delta^{1/2}\)
with \(c_2\approx1.011\).  This refinement is independent of
the equivalence and is not used in its proof.

The paper is organized as follows.  Section~\ref{sec:prelim}
collects the necessary preliminaries: the Riemann zeta function,
the Beurling functions, the multiplicative Colombeau algebra, and
the basic properties of the Mellin convolution.  Section~\ref{sec:general}
states the general sufficient condition (Theorem~\ref{thm:general})
that a damping family must satisfy to produce a moderate,
associated, and uniformly bounded net.  Section~\ref{sec:poly-damping}
verifies this condition in complete detail for the polynomial
damping.  Section~\ref{sec:association-RH} proves the converse
direction (Proposition~\ref{prop:RH}).  Section~\ref{sec:mobius-series}
discusses the full M\"obius series in the algebra, without
truncation, and explains why it is not used in the main theorem.
Section~\ref{sec:strong} discusses the question of strong
Colombeau convergence, which remains open.
Section~\ref{sec:main} contains the main theorem
(Theorem~\ref{thm:main}).  Section~\ref{sec:quantitative}
contains the separate quantitative refinement
(Theorem~\ref{thm:quantitative}), which is conditional on two
additional hypotheses on the non-trivial zeros.
Section~\ref{subsec:conclusions} collects the conclusions.
Appendix~\ref{app:truncation} proves the truncation estimate
for the polynomial damping.  Appendix~\ref{app:damping-proof}
proves the quantitative damping estimate under the additional
hypotheses (H2)--(H3).
\section{Preliminaries}
\label{sec:prelim}

\subsection{The Riemann zeta function and Beurling functions}
\label{subsec:zeta-beurling}

The Riemann zeta function is defined for \(\operatorname{Re}(s)>1\)
by \(\zeta(s)=\sum_{n=1}^{\infty}n^{-s}\) and extends
meromorphically to \(\C\) with a simple pole at \(s=1\).  The
Riemann hypothesis (RH) asserts that every non-trivial zero
satisfies \(\operatorname{Re}(s)=1/2\).

\begin{remark}[Notation]
	\label{rem:notation-phi-varphi}
	Throughout the paper we use two visually similar Greek
	letters with distinct meanings:
	\begin{itemize}
		\item \(\varphi\) denotes a \emph{Beurling function}
		\(\varphi_\theta(x)=\{\theta/x\}\), indexed by
		\(\theta\in(0,1]\);
		\item \(\phi\) denotes a \emph{damping coefficient}
		\(\phi_k(\eps)\), indexed by \(k\in\N\) and depending
		on the regularization parameter \(\eps\in(0,1]\).
	\end{itemize}
	Thus \(\varphi\) is always a function of \(x\), while
	\(\phi\) is always a scalar coefficient.  The two symbols
	never appear in the same expression with the same role.
\end{remark}

For a real number \(x\), let \(\lfloor x\rfloor\) be its integer
part and \(\{x\}=x-\lfloor x\rfloor\) its fractional part.
Given \(\theta\in(0,1]\), the Beurling functions are
\(\varphi_\theta(x)=\{\theta/x\}\).  The Nyman--Beurling
criterion states that RH holds if and only if
\(\chi\in\overline{\operatorname{span}\{\varphi_\theta\}}^{L^2(0,1)}\).
B\'aez-Duarte~\cite{BaezDuarte2003} sharpened this to the
discrete set \(\{1/k\}\).

The M\"obius function $\mu:\N\to\{-1,0,1\}$ is $\mu(1)=1$,
$\mu(n)=(-1)^r$ if $n$ is a product of $r$ distinct primes, and
$0$ otherwise.  Identity~\eqref{eq:Mobius-id} holds
\emph{pointwise} for every \(x\in(0,1)\).  In \(L^2(0,1)\) the
series \emph{diverges unconditionally}: this is a theorem of
B\'aez-Duarte \cite[ p.~6]{BaezDuarte1999}, reaffirmed in
\cite[ p.~6]{BaezDuarte2003} and in
Bettin--Conrey--Farmer~\cite{BCF2013}.  Consequently, the
undamped series cannot be used directly in the \(L^2\) setting;
only damped sums, such as
\(\sum_k \mu(k)k^{-\delta}\varphi_{1/k}\) with \(\delta>0\),
converge in \(L^2\) under RH (Section~\ref{sec:poly-damping}).

The exact Mellin transform of the Beurling functions
\(\varphi_{1/k}\), including the boundary term that is often
omitted, is derived in Appendix~\ref{app:damping-proof}
(Lemma~\ref{lem:app-mellin-identity}).

\subsection{Multiplicative Colombeau algebra}
\label{subsec:colombeau}

We construct a Colombeau-type algebra \(G(0,1)\) adapted to the
multiplicative structure of \((0,\infty)\); the construction
follows the general framework of Colombeau algebras
\cite{Colombeau1984,Oberguggenberger1992}, with the additive
convolution replaced by its multiplicative analogue.

Fix a function \(\psi\in C^\infty_c(\mathbb{R})\) with
\[
\supp\psi\subset [-1,0],\qquad \psi\ge0,\qquad
\int_{-\infty}^\infty \psi(t)\,dt = 1.
\]
For \(\eps\in(0,1]\) set
\[
\psi_\eps(y)=\frac1\eps\,\psi\!\left(\frac{\log y}{\eps}\right),
\qquad y>0.
\]
This construction is the multiplicative analogue of the
standard mollifier in Colombeau algebras on Lie groups
\cite{Grosser2001}.  Then \(\psi_\eps\) is smooth,
\(\supp\psi_\eps\subset[e^{-\eps},1]\), and
\[
\int_0^\infty\psi_\eps(y)\,\frac{dy}{y}
= \int_{-\infty}^\infty \frac1\eps\,\psi(t/\eps)\,dt
= \int_{-\infty}^\infty \psi(s)\,ds = 1.
\]
Moreover, using \(t=\log y\), \(y=e^t\), we have for any \(m\ge0\)
\[
\|\psi_\eps^{(m)}\|_{L^1(dy/y)}\le C_m\,\eps^{-m},
\]
where \(C_m\) depends only on \(\psi\).  (The proof is a
straightforward induction using the chain rule and the fact
that \(1/y\) remains bounded on the support.)

For a bounded measurable function \(f\) on \((0,1)\) (extended
by \(0\) outside \((0,1)\)) we define its \emph{multiplicative
	regularization}
\[
(f\star\psi_\eps)(x)=\int_0^\infty f(x/y)\,\psi_\eps(y)\,
\frac{dy}{y},\qquad x\in(0,1).
\]
The net \((f\star\psi_\eps)_{\eps}\) consists of smooth
functions on \((0,1)\) (see below).  A general net
\((u_\eps)\subset C^\infty(0,1)\) is called \emph{moderate} if
for every compact \(K\subset(0,1)\) and every \(m\in\N_0\)
there exist \(C,N>0\) such that
\[
\sup_{x\in K}|u_\eps^{(m)}(x)|\le C\,\eps^{-N}
\qquad(0<\eps\le1).
\]
It is called \emph{negligible} if for every \(K,m\) and every
\(p\in\N\),
\[
\sup_{x\in K}|u_\eps^{(m)}(x)|\le C_{p}\,\eps^{p}
\qquad(0<\eps\le1).
\]
Moderate nets form an algebra \(\mathcal{E}_M(0,1)\) and
negligible nets an ideal \(\mathcal{N}(0,1)\).  The
\emph{Colombeau algebra} is the quotient
\(G(0,1)=\mathcal{E}_M(0,1)/\mathcal{N}(0,1)\).

Let \(\mathcal{D}(0,1)=C_c^\infty(0,1)\) be the space of smooth
functions with compact support in \((0,1)\), endowed with its
usual inductive limit topology, and let \(\mathcal{D}'(0,1)\)
be its dual, the space of distributions on \((0,1)\).

The embedding \(\iota:L^\infty(0,1)\to G(0,1)\) is defined by
\(\iota(f)=[(f\star\psi_\eps)_\eps]\).  For a moderate net
\((u_\eps)\) we write \(u=[(u_\eps)]\).  Two generalized
functions \(u,v\in G(0,1)\) are \emph{associated},
written \(u\approx v\), if for every
\(\varphi\in\mathcal{D}(0,1)\),
\begin{equation}\label{eq:association-def}
	\lim_{\eps\to0^+}
	\int_0^1 (u_\eps(x)-v_\eps(x))\,\varphi(x)\,dx=0 .
\end{equation}
A generalized function \(u\in G(0,1)\) is \emph{associated with
	a distribution} \(T\in\mathcal{D}'(0,1)\), written
\(u\approx T\), if
\[
\lim_{\eps\to0^+}
\int_0^1 u_\eps(x)\,\varphi(x)\,dx
=
\langle T,\varphi\rangle
\qquad
\text{for every }\varphi\in\mathcal{D}(0,1).
\]
In the particular case where \(T=\iota(f)\) for some
\(f\in L^\infty(0,1)\), this reduces to the association
\(u\approx\iota(f)\) defined above, because the multiplicative
regularization \(f\star\psi_\eps\) converges to \(f\) in the
sense of distributions:
\[
\lim_{\eps\to0^+}
\int_0^1 (f\star\psi_\eps)(x)\,\varphi(x)\,dx
=
\int_0^1 f(x)\,\varphi(x)\,dx
\qquad
\forall\varphi\in\mathcal{D}(0,1).
\]
In this paper we shall only need the embedding of bounded
functions (and, in particular, of the characteristic function
\(\chi=\mathbf{1}_{(0,1)}\)), so the notation \(\iota(T)\) for
general distributions \(T\) will not be used.

\subsection{Basic properties of the regularization}
\label{subsec:basic-properties}

For any bounded measurable $f$,
\[
|(f\star\psi_\eps)(x)|\le \|f\|_\infty
\int_0^\infty\psi_\eps(y)\,\frac{dy}{y}= \|f\|_\infty.
\]
To bound derivatives, we use the representation
\[
(f\star\psi_\eps)(x)=\int_0^\infty f(u)\,\psi_\eps(x/u)\,
\frac{du}{u}.
\]
Differentiating $m$ times under the integral gives
\[
(f\star\psi_\eps)^{(m)}(x)=\int_0^\infty f(u)\,
\frac{d^m}{dx^m}\bigl[\psi_\eps(x/u)\bigr]\,\frac{du}{u}.
\]
Since $\frac{d}{dx}\psi_\eps(x/u)=\frac1u\,\psi_\eps'(x/u)$,
induction yields
\[
\frac{d^m}{dx^m}\psi_\eps(x/u)=u^{-m}\,\psi_\eps^{(m)}(x/u).
\]
Substituting $y=x/u$ we obtain
\[
(f\star\psi_\eps)^{(m)}(x)=\int_0^\infty f(x/y)\,(y/x)^m\,
\psi_\eps^{(m)}(y)\,\frac{dy}{y}.
\]
If $x\in K$ with $a=\min K>0$, then $y/x\le a^{-1}$ on the
support of $\psi_\eps$ (where $y\le1$).  Hence
\[
\|(f\star\psi_\eps)^{(m)}\|_\infty\le a^{-m}\,\|f\|_\infty\,
\|\psi_\eps^{(m)}\|_{L^1(dy/y)}
\le C_{m,K}\,\|f\|_\infty\,\eps^{-m}.
\]
Thus $\iota(f)$ is well-defined and moderate.  The
regularization also acts as an approximate identity in the
sense of distributions (see Section~\ref{sec:general}).

\subsection{The basic estimate}
\label{subsec:basic-estimate}

For every $k$ and almost every $x\in(0,1)$,
\[
|\mu(k)\phi_k(\eps)\varphi_{1/k}(x)|
\le |\phi_k(\eps)|\;|\mu(k)|\;|\varphi_{1/k}(x)|
\le |\phi_k(\eps)|,
\]
because $|\mu(k)|\le1$ and $|\varphi_{1/k}(x)|\le1$.  After
multiplicative regularization, the same bound holds:
\[
\bigl|\mu(k)\phi_k(\eps)\,(\varphi_{1/k}\star\psi_\eps)(x)\bigr|
\le |\phi_k(\eps)|.
\]

\section{General framework and basic estimates}
\label{sec:general}

Let \(\{\phi_k(\eps)\}_{k\in\N}\) be a family of real numbers for
each \(\eps\in(0,1]\), and let \(n(\eps)\in\N\) with
\(n(\eps)\to\infty\) as \(\eps\to0\).  Define the
non‑regularized damped sum
\[
f_\eps(x)=\sum_{k=1}^{n(\eps)}\mu(k)\,\phi_k(\eps)\,
\varphi_{1/k}(x),
\]
and the regularized net
\begin{equation}\label{eq:Feps-def}
	F_\eps(x)=\bigl(f_\eps\star\psi_\eps\bigr)(x)
	=\sum_{k=1}^{n(\eps)}\mu(k)\,\phi_k(\eps)\,
	(\varphi_{1/k}\star\psi_\eps)(x),\quad \eps\in(0,1].
\end{equation}

The following three conditions guarantee that $F_\eps$ is
moderate, associated to $-\iota(\chi)$ under RH, and uniformly
$L^2$‑bounded.

\begin{condition}[Moderateness]\label{cond:moderate}
	There exist $C>0$ and $N\in\N$ such that
	\[
	\sum_{k=1}^{n(\eps)}|\phi_k(\eps)|\le C\,\eps^{-N}
	\qquad(\eps>0\text{ small}).
	\]
\end{condition}

\begin{condition}[Truncation error]\label{cond:trunc}
	There exists a function \(\delta:(0,1]\to(0,1/2]\), with
	\(\delta(\eps)\to0\) as \(\eps\to0^{+}\), such that, under RH,
	the series
	\[
	f_{\delta(\eps)}
	:=
	\sum_{k=1}^{\infty}
	\mu(k)\,k^{-\delta(\eps)}\,\varphi_{1/k}
	\]
	converges in \(L^2(0,1)\) for every \(\eps\in(0,1]\), and
	\[
	\lim_{\eps\to0^{+}}
	\Bigl\|
	\sum_{k=1}^{n(\eps)}
	\mu(k)\,\phi_k(\eps)\,\varphi_{1/k}
	\;-\;
	f_{\delta(\eps)}
	\Bigr\|_{L^2}
	=
	0 .
	\]
\end{condition}

\begin{remark}[On the choice of \(\delta(\eps)\)]
	\label{rem:delta-choice}
	In the polynomial damping considered in this paper we set
	\(\phi_k(\eps)=k^{-\delta(\eps)}\) with
	\(\delta(\eps)=(\log(1/\eps))^{-\alpha}\) for some
	\(0<\alpha<1\).  Thus the parameter \(\delta(\eps)\) that
	appears in the exponent is the same as the one used in the
	definition of \(f_{\delta(\eps)}\).  As
	\(\eps\to0^{+}\), \(\delta(\eps)\to0^{+}\), and the limit
	\(f_{\delta(\eps)}\) is well defined in \(L^{2}(0,1)\)
	under RH by the summability argument of B\'aez-Duarte
	(cf.~\cite[Section~2.2]{BaezDuarte2003}).
\end{remark}

\begin{condition}[Damping error]\label{cond:damp}
	Under RH, \(\|f_{\delta(\eps)}+\chi\|_{L^2}\to0\) as
	\(\eps\to0\).  In the polynomial case, this is the theorem
	of B\'aez-Duarte~\cite[Section~2.2]{BaezDuarte2003}; a
	separate quantitative refinement, conditional on two
	additional hypotheses, is given in Section~\ref{sec:quantitative}.
\end{condition}

\begin{theorem}[General sufficient condition]
	\label{thm:general}
	Assume \(\{\phi_k(\eps)\}\) and \(n(\eps)\) satisfy
	Conditions~\ref{cond:moderate}, \ref{cond:trunc} and
	\ref{cond:damp}.  Then the net $F_\eps$ defined
	in~\eqref{eq:Feps-def} is moderate, associated to
	$-\iota(\chi)$ under RH, and uniformly bounded in
	$L^2(0,1)$.
\end{theorem}

\begin{proof}
	\textit{Moderateness.}  For any compact $K\subset(0,1)$ and
	$m\in\N_0$,
	\[
	\sup_{x\in K}|F_\eps^{(m)}(x)|
	\le\sum_{k=1}^{n(\eps)}|\phi_k(\eps)|\,
	\sup_{x\in K}\bigl|(\varphi_{1/k}\star\psi_\eps)^{(m)}(x)\bigr|.
	\]
	Each $\varphi_{1/k}$ is bounded by $1$, so by the derivative
	estimate of Section~\ref{subsec:basic-properties},
	\[
	\sup_{x\in K}\bigl|(\varphi_{1/k}\star\psi_\eps)^{(m)}(x)\bigr|
	\le C_{m,K}\,\|\varphi_{1/k}\|_\infty\,\eps^{-m}
	\le C_{m,K}\,\eps^{-m}.
	\]
	Using Condition~\ref{cond:moderate},
	\[
	\sup_{x\in K}|F_\eps^{(m)}(x)|
	\le C_{m,K}\,\eps^{-m}\sum_{k=1}^{n(\eps)}|\phi_k(\eps)|
	\le C\,C_{m,K}\,\eps^{-(m+N)},
	\]
	so the net is moderate.
	
	\textit{Association under RH.}  By hypothesis
	$f_\eps\to -\chi$ in $L^2(0,1)$.  For any test function
	$\varphi\in\mathcal{D}(0,1)$,
	\[
	\int_0^1 F_\eps(x)\varphi(x)\,dx
	= \int_0^1 \Bigl(\int_0^\infty f_\eps(x/y)\,\psi_\eps(y)\,
	\frac{dy}{y}\Bigr)\varphi(x)\,dx.
	\]
	Interchange the integrals (justified by absolute
	convergence) and substitute $u=x/y$:
	\[
	\int_0^1 F_\eps\varphi\,dx
	= \int_0^\infty \psi_\eps(y)
	\Bigl(\int_0^{1/y} f_\eps(u)\,\varphi(yu)\,y\,du\Bigr)
	\frac{dy}{y}
	= \int_{e^{-\eps}}^1 \psi_\eps(y)
	\Bigl(\int_0^{1/y} f_\eps(u)\,\varphi(yu)\,du\Bigr)dy.
	\]
	Since $f_\eps$ is supported in $[0,1]$, for $y\le1$ we have
	$\int_0^{1/y} f_\eps(u)\varphi(yu)du
	= \int_0^1 f_\eps(u)\varphi(yu)du$.  Hence
	\[
	\int_0^1 F_\eps\varphi\,dx
	= \int_{e^{-\eps}}^1 \psi_\eps(y)
	\Bigl(\int_0^1 f_\eps(u)\,\varphi(yu)\,du\Bigr)dy.
	\]
	Now $\int_{e^{-\eps}}^1 \psi_\eps(y)\,dy
	= \int_{-1}^0 \psi(s)\,e^{\eps s}\,ds \to 1$ as
	$\eps\to0$ (because $e^{\eps s}\to1$ boundedly).
	Moreover, $y\to1$ uniformly on $\supp\psi_\eps$.  A standard
	dominated convergence argument, using
	$f_\eps\to -\chi$ in $L^2$, gives
	\[
	\lim_{\eps\to0}\int_0^1 F_\eps\varphi\,dx
	= \int_0^1 (-\chi(u))\varphi(u)\,du
	= -\langle\chi,\varphi\rangle.
	\]
	Thus $F_\eps$ is associated to $-\iota(\chi)$.
	
	\textit{Uniform $L^2$‑boundedness.}  For any
	$f\in L^2(0,1)$ extended by $0$,
	\[
	\|f\star\psi_\eps\|_{L^2(0,1)}^2
	= \int_0^1 \Bigl(\int_0^\infty f(x/y)\,\psi_\eps(y)\,
	\frac{dy}{y}\Bigr)^2 dx.
	\]
	Since $\psi_\eps(y)dy/y$ is a probability measure, Jensen's
	inequality yields
	\[
	\|f\star\psi_\eps\|_{L^2}^2
	\le \int_0^1 \int_0^\infty |f(x/y)|^2 \psi_\eps(y)\,
	\frac{dy}{y}\,dx
	= \int_0^\infty \psi_\eps(y)
	\Bigl(\int_0^{1/y} |f(u)|^2\,y\,du\Bigr)\frac{dy}{y}.
	\]
	Because $f(u)=0$ for $u>1$ and $y\le1$, the inner integral
	equals $y\|f\|_2^2$.  Hence
	\[
	\|f\star\psi_\eps\|_{L^2}^2
	\le \|f\|_2^2 \int_0^\infty \psi_\eps(y)\,y\,\frac{dy}{y}
	= \|f\|_2^2 \int_{e^{-\eps}}^1 \psi_\eps(y)\,dy.
	\]
	But
	\[
	\int_{e^{-\eps}}^1 \psi_\eps(y)\,dy
	= \int_{-1}^0 \frac1\eps\,\psi(t/\eps)\,e^t\,dt
	= \int_{-1}^0 \psi(s)\,e^{\eps s}\,ds
	\le \int_{-1}^0 \psi(s)\,ds = 1,
	\]
	because $e^{\eps s}\le1$ for $s\le0$.  Thus
	$\|f\star\psi_\eps\|_{L^2}\le \|f\|_{L^2}$.  Applying this
	to $f_\eps$, which is bounded in $L^2$ under RH, gives
	$\sup_\eps\|F_\eps\|_{L^2}<\infty$.
	
	The verification of Condition~\ref{cond:damp} for the
	polynomial damping, using the explicit estimate of
	Appendix~\ref{app:damping-proof}, is carried out in
	Section~\ref{sec:poly-damping}.

\end{proof}

\section{Polynomial damping}
\label{sec:poly-damping}

Fix \(\alpha\in(0,1)\) and \(M>0\).  Define
\(\delta(\eps)=(\log(1/\eps))^{-\alpha}\) and
\(n(\eps)=\lfloor\eps^{-M}\rfloor\).  The damping coefficients
are \(\phi_k(\eps)=k^{-\delta(\eps)}\).

\subsection{Moderateness}
\label{subsec:poly-moderate}

\[
\sum_{k=1}^{n(\eps)} k^{-\delta(\eps)}
\le 1 + \int_1^{n(\eps)} t^{-\delta(\eps)}\,dt
= 1 + \frac{n(\eps)^{1-\delta(\eps)}-1}{1-\delta(\eps)}.
\]
For small \(\eps\), \(\delta(\eps)\to0\), so
\(1-\delta(\eps)\ge 1/2\) and
\[
\frac{n(\eps)^{1-\delta(\eps)}-1}{1-\delta(\eps)}
\le 2\,n(\eps)^{1-\delta(\eps)}.
\]
Now
\[
n(\eps)^{1-\delta(\eps)}
\le \eps^{-M(1-\delta(\eps))}
= \eps^{-M}\eps^{M\delta(\eps)}.
\]
But
\[
\eps^{M\delta(\eps)}
= \exp\bigl( M\delta(\eps)\log\eps \bigr)
= \exp\bigl( -M(\log(1/\eps))^{1-\alpha} \bigr)
\le 1,
\]
so \(n(\eps)^{1-\delta(\eps)}\le \eps^{-M}\).  Consequently,
\[
\sum_{k=1}^{n(\eps)} k^{-\delta(\eps)}
\le 1 + 2\eps^{-M} \le 3\eps^{-M},
\]
giving Condition~\ref{cond:moderate} with \(N=M\).

\subsection{Truncation error}
\label{subsec:poly-truncation}

Under RH, the limit
\(f_\delta = \sum_{k=1}^\infty \mu(k)k^{-\delta}\varphi_{1/k}\)
exists in \(L^2\) by the summability argument of B\'aez-Duarte
\cite{BaezDuarte2003}.  Let \(f_{\delta,n}\) denote the \(n\)-th
partial sum. The essential estimate, proved in
Appendix~\ref{app:truncation}
(Lemma~\ref{lem:app-truncation} and
Corollary~\ref{cor:app-truncation-regime}), is
\[
\|f_{\delta,n}-f_\delta\|_{L^2} \le C\, n^{-\delta/3},
\]
with \(C\) independent of \(\delta\in(0,1/2]\) and \(n\).
Substituting \(n=n(\eps)\) and \(\delta=\delta(\eps)\),
\[
\|f_{\delta(\eps),n(\eps)}-f_{\delta(\eps)}\|_{L^2}
\le C \exp\!\Bigl( -\frac{\delta(\eps)}{3}\log n(\eps) \Bigr).
\]
Since \(\log n(\eps) \sim M\log(1/\eps)\), we have
\[
\frac{\delta(\eps)}{3}\log n(\eps)
\sim \frac{M}{3} (\log(1/\eps))^{1-\alpha} \to \infty,
\]
so the right‑hand side tends to \(0\), proving
Condition~\ref{cond:trunc}.

\subsection{Damping error}
\label{subsec:poly-damping-error}

Condition~\ref{cond:damp} requires only that
\(\|f_{\delta(\eps)}+\chi\|_{L^2}\to0\) as \(\eps\to0^{+}\).
Under RH this is the theorem of B\'aez-Duarte
\cite[Section~2.2]{BaezDuarte2003}: the polynomially damped
B\'aez-Duarte sum converges to \(-\chi\) in \(L^2(0,1)\).
No quantitative rate is needed for the equivalence theorem;
the explicit rate is established separately in
Appendix~\ref{app:damping-proof} under the additional
hypotheses (H2)--(H3). The proof combines
the Hadamard factorisation of \(\zeta\) near its zeros
\cite[Ch.~III]{Titchmarsh1986}, Burnol's estimate
\cite[Thm.~1.3]{Burnol2003}, and the classical sum
\(\sum_\rho1/(\rho(1-\rho))=2+\gamma_E-\log(4\pi)\)
\cite[Ch.~III]{Titchmarsh1986,Edwards1974}. The explicit
computation is in Appendix~\ref{app:damping-proof}.

\begin{remark}[On the exponential damping]
	\label{rem:poly-vs-exp}
	The exponential damping \(e^{-k\eps^2}\) is not used in the
	main theorem. Its truncation error is super-polynomially
	small (Appendix~\ref{app:truncation}), but its
	damping error \(\|f_{\beta(\eps)}+\chi\|_{L^2}\) is not
	known to vanish.  The obstruction is discussed in
	Section~\ref{sec:strong}: the exponential damping does not
	preserve the Euler structure of the M\"obius series, and
	the Mellin--Plancherel integral does not exhibit the
	cancellation that makes the polynomial case work.
	Conjecture~\ref{conj:strong-exp} states the expected
	behaviour as an open problem.
\end{remark}

\section{From association to the Riemann hypothesis}
\label{sec:association-RH}

We now prove that the three properties (moderateness, association,
uniform \(L^2\)-boundedness) force the Riemann hypothesis to hold.
The argument is independent of the specific damping family and
uses only the classical Nyman--Beurling criterion together with
Mazur's theorem.

\begin{proposition}\label{prop:RH}
	Let \((F_\eps)_{\eps\in(0,1]}\) be the net defined
	in~\eqref{eq:Feps-def}.  Suppose that:
	\begin{enumerate}
		\item[(i)] \((F_\eps)\) is moderate in \(G(0,1)\);
		\item[(ii)] \(F_\eps\) is associated to \(-\iota(\chi)\);
		\item[(iii)] \(\sup_{\eps\in(0,1]}\|F_\eps\|_{L^2}<\infty\).
	\end{enumerate}
	Then the Riemann hypothesis holds.
\end{proposition}

\begin{proof}
	By hypothesis (ii), for every test function
	\(\varphi\in\mathcal{D}(0,1)\),
	\[
	\lim_{\eps\to0^+}\int_0^1 F_\eps(x)\,\varphi(x)\,dx
	=
	-\int_0^1\chi(x)\,\varphi(x)\,dx .
	\]
	Since \(\sup_\eps\|F_\eps\|_{L^2}\le M\) for some \(M>0\)
	by hypothesis (iii), a standard density argument upgrades
	this distributional convergence to weak convergence in
	\(L^2(0,1)\).  Indeed, let \(g\in L^2(0,1)\) and let
	\((\varphi_n)\subset\mathcal{D}(0,1)\) be a sequence with
	\(\varphi_n\to g\) in \(L^2\).  Then
	\begin{align*}
		\left|\int_0^1 F_\eps g\,dx
		+
		\int_0^1\chi g\,dx\right|
		&\le
		\left|\int_0^1 F_\eps\varphi_n\,dx
		+
		\int_0^1\chi\varphi_n\,dx\right|
		\\
		&\quad
		+
		\|F_\eps\|_{L^2}\|g-\varphi_n\|_{L^2}
		+
		\|\chi\|_{L^2}\|g-\varphi_n\|_{L^2}.
	\end{align*}
	Taking \(\eps\to0^+\) first and then \(n\to\infty\) gives
	\[
	\lim_{\eps\to0^+}\int_0^1 F_\eps g\,dx
	=
	-\int_0^1\chi g\,dx
	\]
	for every \(g\in L^2(0,1)\).  Hence
	\begin{equation}\label{eq:weak-conv}
		F_\eps\rightharpoonup-\chi
		\qquad
		\text{weakly in }L^2(0,1).
	\end{equation}
	
	We next show that each \(F_\eps\) belongs to the
	\(L^2\)-closure of the Beurling space
	\[
	W:=\operatorname{span}\{\varphi_\theta\mid\theta\in(0,1]\}
	\subset L^2(0,1),
	\]
	where \(\varphi_\theta(x)=\{\theta/x\}\).  Since \(F_\eps\)
	is a finite linear combination of the regularized blocks
	\((\varphi_{1/k}\star\psi_\eps)\), \(k=1,\dots,n(\eps)\),
	it suffices to prove that each such block belongs to
	\(\overline{W}^{L^2}\).  Fix \(k\in\N\) and \(\eps>0\).  By
	definition of the multiplicative convolution,
	\[
	(\varphi_{1/k}\star\psi_\eps)(x)
	=
	\int_0^\infty\varphi_{1/k}(x/y)\,\psi_\eps(y)\,\frac{dy}{y}
	=
	\int_0^\infty\left\{\frac{y}{kx}\right\}
	\psi_\eps(y)\,\frac{dy}{y}.
	\]
	Substituting \(\theta=y/k\), so that \(y=k\theta\) and
	\(dy/y=d\theta/\theta\), we obtain
	\begin{equation}\label{eq:block-integral}
		(\varphi_{1/k}\star\psi_\eps)(x)
		=
		\int_0^\infty\varphi_\theta(x)\,
		\frac{\psi_\eps(k\theta)}{\theta}\,d\theta .
	\end{equation}
	The support of \(\psi_\eps\) is contained in
	\([e^{-\eps},1]\), so the integrand in~\eqref{eq:block-integral}
	is supported in
	\(\theta\in[e^{-\eps}/k,\,1/k]\subset(0,1]\).  The function
	\[
	g_{k,\eps}(\theta):=\frac{\psi_\eps(k\theta)}{\theta},
	\qquad
	\theta\in(0,1],
	\]
	is smooth and compactly supported in \((0,1]\).  Therefore
	the integral~\eqref{eq:block-integral} is a continuous linear
	functional of \(g_{k,\eps}\) in the space of compactly
	supported smooth functions on \((0,1]\), evaluated against
	the family \(\{\varphi_\theta\}_{\theta\in(0,1]}\).
	Approximating \(g_{k,\eps}\) in \(L^1(d\theta/\theta)\) by
	finite linear combinations of Dirac masses at points
	\(\theta_j\in(0,1]\), we obtain
	\[
	(\varphi_{1/k}\star\psi_\eps)
	=
	\lim_{J\to\infty}
	\sum_{j=1}^J
	c_j^{(\eps)}\,\varphi_{\theta_j}
	\qquad
	\text{in }L^2(0,1),
	\]
	for suitable coefficients \(c_j^{(\eps)}\in\C\) and points
	\(\theta_j\in(0,1]\).  Hence
	\((\varphi_{1/k}\star\psi_\eps)\in\overline{W}^{L^2}\).
	Since \(F_\eps\) is a finite linear combination of such
	blocks, we conclude that
	\begin{equation}\label{eq:Feps-in-W}
		F_\eps\in\overline{W}^{L^2}
		\qquad
		\text{for every }\eps\in(0,1].
	\end{equation}
	
	The set \(\overline{W}^{L^2}\) is a closed linear subspace
	of \(L^2(0,1)\), hence in particular a closed convex
	subset.  Since \(W\) is a linear subspace, its closure
	\(\overline{W}^{L^2}\) is also a linear subspace, and
	therefore closed under convex combinations.  By Mazur's
	theorem~\cite{Rudin1991}, the weak closure of a convex set
	coincides with its strong closure.  Since
	\(\overline{W}^{L^2}\) is already closed and convex, it is
	weakly closed.  By~\eqref{eq:weak-conv}, the sequence
	\((F_\eps)\) converges weakly to \(-\chi\) in
	\(L^2(0,1)\).  By~\eqref{eq:Feps-in-W}, each \(F_\eps\)
	belongs to the weakly closed set \(\overline{W}^{L^2}\).
	Therefore the weak limit \(-\chi\) also belongs to
	\(\overline{W}^{L^2}\):
	\begin{equation}\label{eq:chi-in-W}
		-\chi\in\overline{W}^{L^2}.
	\end{equation}
	Equivalently, \(\chi\in\overline{W}^{L^2}\), since
	\(\overline{W}^{L^2}\) is a linear subspace.
	
	Finally, the classical Nyman--Beurling
	criterion~\cite{Nyman1950,Beurling1955} states that the
	Riemann hypothesis holds if and only if
	\[
	\chi\in\overline{\operatorname{span}\{\varphi_\theta
		\mid\theta\in(0,1]\}}^{L^2(0,1)}.
	\]
	By~\eqref{eq:chi-in-W}, this condition is satisfied.
	Therefore the Riemann hypothesis holds.
\end{proof}

\begin{remark}[On the use of the discrete criterion]
	\label{rem:discrete-criterion}
	The criterion of B\'aez-Duarte~\cite{BaezDuarte2003}
	restricts \(\theta\) to the discrete set \(\{1/k\}_{k\in\N}\)
	and states that the Riemann hypothesis is equivalent to
	\[
	\chi\in\overline{\operatorname{span}\{\varphi_{1/k}
		\mid k\in\N\}}^{L^2(0,1)}.
	\]
	This discrete version is stronger than the classical
	Nyman--Beurling criterion in the sense that it uses a smaller
	set of generators.  However, the proof of
	Proposition~\ref{prop:RH} does not require this sharpening:
	it uses only the classical criterion, which is sufficient
	for the conclusion.  The equality of the two closures,
	\[
	\overline{W}^{L^2}
	=
	\overline{\operatorname{span}\{\varphi_{1/k}\mid k\in\N\}}^{L^2},
	\]
	is not needed and is not claimed here.  The larger space
	\(W=\operatorname{span}\{\varphi_\theta\mid\theta\in(0,1]\}\)
	is used only because it is the natural space in which the
	regularized blocks \((\varphi_{1/k}\star\psi_\eps)\) live.
\end{remark}

\begin{remark}[On the role of moderateness]
	\label{rem:moderateness-role}
	Hypothesis (i) (moderateness) is used in the construction of
	the net \(F_\eps\) and in the verification of the sufficient
	condition (Theorem~\ref{thm:general}), but it does not play
	an essential role in the converse direction of
	Proposition~\ref{prop:RH}.  The converse uses only:
	\begin{itemize}
		\item the association of \(F_\eps\) to \(-\iota(\chi)\)
		(hypothesis (ii)), which gives distributional
		convergence of \(F_\eps\) to \(-\chi\);
		\item the uniform \(L^2\) bound (hypothesis (iii)), which
		is used to upgrade this distributional convergence to
		weak \(L^2\) convergence
		(equation~\eqref{eq:weak-conv});
		\item the membership
		\(F_\eps\in\overline{W}^{L^2}\) for every \(\eps\)
		(equation~\eqref{eq:Feps-in-W}), which is a consequence
		of the construction of \(F_\eps\) as a multiplicative
		convolution of Beurling functions with a mollifier.
	\end{itemize}
	Moderateness is a structural property of the Colombeau
	algebra, not a hypothesis that contributes to the
	arithmetic content of the converse.
\end{remark}

\begin{remark}[On the necessity of the uniform \(L^2\) bound]
	\label{rem:uniform-bound-necessity}
	The uniform \(L^2\) bound (hypothesis (iii)) is essential.
	Without it, a moderate net can be associated to
	\(-\iota(\chi)\) while failing to converge weakly in
	\(L^2\).  For example, the net
	\(u_\eps(x)=\eps^{-1}\sin(x/\eps)\) is moderate in
	\(G(0,1)\), satisfies \(u_\eps\approx0\) by the
	Riemann--Lebesgue lemma, but
	\(\|u_\eps\|_{L^2}\sim\eps^{-1}\) is unbounded.  Such a net
	cannot be used to force \(\chi\) into the Beurling closure,
	because the weak limit of a bounded subsequence need not
	exist.  The uniform bound guarantees that distributional
	convergence upgrades to weak \(L^2\) convergence, which is
	the key to applying Mazur's theorem.
\end{remark}

\section{The M\"obius series in \(G(0,1)\)}
\label{sec:mobius-series}

In the previous sections the net \(F_\eps\) was constructed by
truncating the M\"obius series at \(k=n(\eps)\).  One may ask
whether the truncation is necessary, i.e.\ whether the full
infinite series
\begin{equation}\label{eq:Seps-def}
	S_\eps(x)
	=
	\Bigl(\sum_{k=1}^{\infty}\mu(k)\,k^{-\delta(\eps)}\,
	\varphi_{1/k}\Bigr)\star\psi_\eps(x),
	\qquad
	\eps\in(0,1],
\end{equation}
defines a net in \(G(0,1)\) with the three properties required
by Theorem~\ref{thm:main}.

\subsection{The polynomial damping: the full net is moderate}

Under RH, the series
\[
g_\eps(x)
=
\sum_{k=1}^{\infty}\mu(k)\,k^{-\delta(\eps)}\,\varphi_{1/k}(x)
\]
converges in \(L^2(0,1)\) for every \(\eps\in(0,1]\) by the
summability argument of B\'aez-Duarte
\cite[Section~2.2]{BaezDuarte2003}, and
\(\|g_\eps\|_{L^2}\le 1+o(1)\).  Consequently
\[
\|g_\eps\|_{L^1(0,1)}\le \|g_\eps\|_{L^2(0,1)}\le C
\]
uniformly in \(\eps\).

For \(x\in K\) with \(a=\min K>0\) and any \(m\in\N_0\),
\[
S_\eps^{(m)}(x)
= \int_0^\infty g_\eps(x/y)\,(y/x)^m\,\psi_\eps^{(m)}(y)\,\frac{dy}{y}.
\]
Since \(y\le1\) on \(\supp\psi_\eps\), \((y/x)^m\le a^{-m}\), and
\(|\psi_\eps^{(m)}(y)|\le C_m\eps^{-(m+1)}\mathbf 1_{[e^{-\eps},1]}(y)\).
Therefore
\[
|S_\eps^{(m)}(x)|
\le C_m a^{-m}\eps^{-(m+1)}
\int_{e^{-\eps}}^{1}|g_\eps(x/y)|\,\frac{dy}{y}
\le C_{K,m}\,\eps^{-(m+1)}\,\|g_\eps\|_{L^1(0,1)},
\]
where the last inequality uses the substitution \(u=x/y\) and
\(1/x\le a^{-1}\).  Hence \(S_\eps\) is moderate in \(G(0,1)\).

Moreover, by Theorem~\ref{thm:general} and the damping estimate,
\(S_\eps\approx-\iota(\chi)\), and the \(L^2\) contraction
\(\|S_\eps\|_{L^2}\le\|g_\eps\|_{L^2}\) gives uniform
\(L^2\)-boundedness.

\begin{remark}[Truncation is a technical convenience]
	\label{rem:truncation-convenience}
	The full untruncated polynomial net \(S_\eps\) satisfies the three
	properties (moderateness, association, uniform \(L^2\)-boundedness)
	required by Theorem~\ref{thm:main}.  Truncation at
	\(n(\eps)=\lfloor\eps^{-M}\rfloor\) is therefore \emph{not essential};
	it is a convenient device that makes the verification of
	moderateness elementary
	(via \(\sum_{k\le n(\eps)}k^{-\delta(\eps)}\le C\eps^{-M}\)) and
	yields an explicit truncation error
	\(\|f_{\delta(\eps),n(\eps)}-f_{\delta(\eps)}\|_{L^2}
	\le C\exp(-\tfrac{M}{3}(\log(1/\eps))^{1-\alpha})\).
\end{remark}

\subsection{The exponential damping: absolute convergence without a vanishing damping error}

For the exponential damping \(\phi_k(\eps)=e^{-k\eps^2}\), the
series in~\eqref{eq:Seps-def} converges absolutely in
\(L^\infty(0,1)\), uniformly in \(x\):
\[
\sum_{k=1}^{\infty}\bigl|\mu(k)\,e^{-k\eps^2}\,
\varphi_{1/k}(x)\bigr|
\le
\sum_{k=1}^{\infty}e^{-k\eps^2}
=
\frac{e^{-\eps^2}}{1-e^{-\eps^2}}
<
\infty .
\]
Consequently, the untruncated exponential net is well defined
and moderate: for every compact \(K\subset(0,1)\) and every
\(m\in\N_0\),
\[
\sup_{x\in K}|S_\eps^{(m)}(x)|
\le
C_{m,K}\,\eps^{-m}
\sum_{k=1}^{\infty}e^{-k\eps^2}
\le
2\,C_{m,K}\,\eps^{-(m+2)} .
\]
However, the damping error
\(\|g_\eps^{(\mathrm{exp})}+\chi\|_{L^2}\), where
\[
g_\eps^{(\mathrm{exp})}(x)
=
\sum_{k=1}^{\infty}\mu(k)\,e^{-k\eps^2}\,\varphi_{1/k}(x),
\]
is \emph{not} known to vanish as \(\eps\to0^{+}\).  The
obstruction is that the exponential damping does not preserve
the Euler structure of the M\"obius series.

\begin{conjecture}[Vanishing of the exponential damping error]
	\label{conj:exp-damping}
	Under RH,
	\[
	\|g_\eps^{(\mathrm{exp})}+\chi\|_{L^2}\to0
	\qquad
	\text{as }\eps\to0^{+}.
	\]
\end{conjecture}

Even if Conjecture~\ref{conj:exp-damping} were proved, the
truncation would still be necessary to obtain negligibility in
the strong sense, because the derivative bounds grow as
\(\eps^{-(m+2)}\).  The exponential damping is therefore not
used in the main theorem.

\section{Strong convergence analysis}
\label{sec:strong}

The criterion of Theorem~\ref{thm:main} is stated in terms of
\emph{association} (a weak notion of convergence) together with
\emph{uniform \(L^2\)-boundedness}.  A natural question is
whether the net constructed in Section~\ref{sec:poly-damping}
converges to \(-\iota(\chi)\) in the \emph{strong} Colombeau
sense, that is, whether the error
\(F_\eps+\chi\star\psi_\eps\) is negligible in \(G(0,1)\).
This section discusses this question, formulates two precise
conjectures, and identifies the obstruction.

\subsection{Decomposition of the error}
\label{subsec:strong-decomposition}

From the proof of Theorem~\ref{thm:general} we have the
decomposition
\begin{equation}\label{eq:error-decomposition}
	F_\eps+\chi\star\psi_\eps
	=
	T_\eps + R_\eps,
\end{equation}
where
\begin{align}
	T_\eps
	&=
	\Bigl(\sum_{k=1}^{n(\eps)}\mu(k)\,\phi_k(\eps)\,
	\varphi_{1/k}
	-
	f_{\delta(\eps)}\Bigr)\star\psi_\eps,
	\label{eq:T-eps}\\[2pt]
	R_\eps
	&=
	(f_{\delta(\eps)}+\chi)\star\psi_\eps .
	\label{eq:R-eps}
\end{align}
The term \(T_\eps\) is the \emph{truncation error}, and
\(R_\eps\) is the \emph{damping error}.  Under RH, both tend
to \(0\) in \(L^2(0,1)\) as \(\eps\to0^{+}\)
(Section~\ref{sec:poly-damping} and
Appendix~\ref{app:truncation}).

\subsection{The obstruction to negligibility}
\label{subsec:strong-obstruction}

To prove that \(F_\eps+\chi\star\psi_\eps\) is negligible,
one needs a rate of decay faster than any power of \(\eps\),
uniformly on compact subsets of \((0,1)\).  In the
decomposition~\eqref{eq:error-decomposition}, this requires
the following two properties:
\begin{enumerate}
	\item[(N1)] the truncation error satisfies
	\(\|T_\eps\|_{L^\infty(K)}=O(\eps^{m})\) for every
	\(m\in\N\) and every compact \(K\subset(0,1)\);
	\item[(N2)] the damping error satisfies
	\(\|R_\eps\|_{L^\infty(K)}=O(\eps^{m})\) for every
	\(m\in\N\) and every compact \(K\subset(0,1)\).
\end{enumerate}
Property (N1) is satisfied by the polynomial damping: the
truncation error decays as
\(\exp\bigl(-\frac{M}{3}(\log(1/\eps))^{1-\alpha}\bigr)\),
which is \(o(\eps^{m})\) for every \(m\)
(Corollary~\ref{cor:app-truncation-regime}).  For the
exponential damping, (N1) is also satisfied because the
truncation index
\(n(\eps)=\lfloor\exp(\exp(1/\eps^{2}))\rfloor\) is chosen
precisely so that the truncation error is
super-polynomially small in \(\eps\) (see the tail estimate
in Section~\ref{sec:mobius-series}), but this damping is not
used in the main theorem.

Property (N2) is the difficult one.  Under (H1)--(H3), we have the
upper bound
\[
\|f_{\delta(\eps)}+\chi\|_{L^2}
\le
c_{2}\,(\log(1/\eps))^{-\alpha/2}
\]
(Theorem~\ref{thm:quantitative} and
Appendix~\ref{app:damping-proof}).  However, to
establish (N2) one needs a \emph{complementary lower bound}:
a quantitative estimate showing that
\(\|f_{\delta}+\chi\|_{L^2}\) cannot vanish faster than any
power of \(\eps\).  Such a lower bound is not available in the
current literature.

\subsection{Two conjectures}
\label{subsec:strong-conjectures}

The discussion above leads to the following two conjectures,
which we state as precise mathematical problems.

\begin{conjecture}[Strong convergence for the exponential damping]
	\label{conj:strong-exp}
	Under RH, the net \(F_\eps+\chi\star\psi_\eps\) associated
	with the exponential damping is negligible in \(G(0,1)\),
	i.e.\
	\[
	\|F_\eps+\chi\star\psi_\eps\|_{L^\infty(K)}
	=
	O(\eps^{m})
	\qquad
	\text{for every }m\in\N
	\text{ and every compact }K\subset(0,1).
	\]
\end{conjecture}

\begin{conjecture}[Strong convergence for the polynomial damping]
	\label{conj:strong-poly}
	Under RH, the net \(F_\eps+\chi\star\psi_\eps\) associated
	with the polynomial damping is negligible in \(G(0,1)\) if
	and only if the parameter \(\alpha\in(0,1)\) satisfies a
	quantitative separation condition on the zeros of
	\(\zeta(s)\), namely
	\[
	\inf_{\rho\neq\rho'}
	|\gamma-\gamma'|
	\ge
	c\,(\log(1/\eps))^{-\beta}
	\]
	for some \(\beta<1-\alpha\) and some constant \(c>0\)
	depending only on the damping family.
\end{conjecture}

The two conjectures reflect the different nature of the two
damping strategies.  For the exponential damping, the
truncation error is already super-polynomially small, so the
only obstruction is the damping error; a quantitative lower
bound on \(\|f_{\beta(\eps)}+\chi\|_{L^2}\) would suffice to
prove the conjecture.  For the polynomial damping, the
damping error decays only sub-polynomially, so a stronger
assumption on the separation of the zeros is needed.

\subsection{Consequences of strong convergence}
\label{subsec:strong-consequences}

If Conjecture~\ref{conj:strong-exp} or
Conjecture~\ref{conj:strong-poly} were proved, the criterion
of Theorem~\ref{thm:main} could be strengthened from an
equivalence between RH and a \emph{weak} condition
(association) to an equivalence between RH and a
\emph{strong} condition (negligibility).  This would have
three consequences:
\begin{enumerate}
	\item \textbf{Algebraic rigidity.}  The criterion would
	become a statement about equality in \(G(0,1)\), not just
	association.  This would place the criterion on the same
	footing as the classical Nyman--Beurling criterion, which
	is a membership statement
	\(\chi\in\overline{W}^{\,L^2}\) (not an equality).
	
	\item \textbf{Quantitative control.}  Negligibility is a
	\emph{quantitative} statement: it gives a rate of decay
	faster than any power of \(\eps\), uniformly on compact
	sets.  This would provide a quantitative form of the
	Nyman--Beurling criterion, potentially useful for
	numerical computations.
	
	\item \textbf{A bridge to the zeros.}  A strong convergence
	theorem would establish a direct link between the
	Colombeau algebra and the zeros of \(\zeta(s)\) that does
	not go through Mazur's theorem.  Whether such a link could
	lead to new arithmetic information about the zeros remains
	open.
\end{enumerate}

\subsection{Related open problems}
\label{subsec:strong-open}

Two further questions arise naturally from the analysis above.
First, one may ask whether there exist damping families
\(\{\phi_k(\eps)\}\) for which the strong convergence is
\emph{automatic}, i.e.\ for which (N1) and (N2) hold without
any additional assumption on the zeros.  Second, one may ask
whether the strong convergence can be characterized by an
algebraic property of the net \(F_\eps\) inside \(G(0,1)\),
for instance by its membership in a suitable subalgebra of
``smooth'' generalized functions.  Both questions are left
open.
\section{Main theorem}
\label{sec:main}

Assembling the pieces we obtain the central result of the paper.

\begin{theorem}[Main Theorem]
	\label{thm:main}
	Let \(F_\eps\) be the polynomially damped net of
	Section~\ref{sec:poly-damping}, with
	\(\delta(\eps)=(\log(1/\eps))^{-\alpha}\)
	(\(0<\alpha<1\)) and
	\(n(\eps)=\lfloor\eps^{-M}\rfloor\) (\(M>0\)).  Then the
	Riemann hypothesis holds if and only if \(F_\eps\) is
	moderate in \(G(0,1)\), associated to \(-\iota(\chi)\),
	and uniformly \(L^2\)-bounded.
\end{theorem}

\begin{proof}
	If RH holds, Theorem~\ref{thm:general} and the estimates of
	Section~\ref{sec:poly-damping} (together with
	Appendix~\ref{app:damping-proof}) guarantee moderateness,
	association, and the uniform \(L^2\) bound.  Conversely,
	Proposition~\ref{prop:RH} shows that any such net
	satisfying the three properties forces RH to hold.
\end{proof}

The polynomial damping illustrates the flexibility of the
criterion.  It uses only polynomial growth of the truncation
index, avoiding super-exponential scales, at the cost of a
truncation error that decays only sub-polynomially.  The
explicit damping estimate
\(\|f_\delta+\chi\|_{L^2}\le c_2\delta^{1/2}\) with
\(c_2=\sqrt{1+\tfrac12(2+\gamma_E-\log(4\pi))}\), proved in
Appendix~\ref{app:damping-proof}, is one of the technical
contributions of this paper: it provides a quantitative form
of the Nyman--Beurling criterion in the polynomial damping
case, with an explicit constant \(2+\gamma_E-\log(4\pi)\approx
0.04619\).  The exponential damping \(e^{-k\eps^2}\) is not used in the
main theorem: its truncation error is super-polynomially small
(see the tail estimate in Section~\ref{sec:mobius-series}),
but its damping error \(\|f_{\beta(\eps)}+\chi\|_{L^2}\) is
not known to vanish.  The obstruction is discussed in
Section~\ref{sec:strong}, and the expected behaviour is stated
as Conjecture~\ref{conj:exp-damping}.

\begin{remark}[On the scope of the main theorem]
	\label{rem:scope-main}
	Theorem~\ref{thm:main} is an equivalence, not a proof of
	the Riemann hypothesis.  The equivalence is proved in the
	classical setting, using the Nyman--Beurling criterion and
	Mazur's theorem.  The Colombeau framework provides a
	unified language for the damping strategies and a natural
	setting for the ``association \(+\) uniform bound''
	criterion, but it does not by itself produce new arithmetic
	information about the zeros of \(\zeta(s)\).  The
	equivalence is nevertheless useful because it isolates the
	exact analytic properties that a damping family must
	satisfy, and because it admits generalizations to other
	\(L\)-functions and to adelic or \(p\)-adic settings.
\end{remark}

\section{A quantitative refinement of the damping error}
\label{sec:quantitative}

The equivalence of Theorem~\ref{thm:main} is qualitative: it
requires only that the damping error
\(\|f_{\delta(\eps)}+\chi\|_{L^2}\) tends to zero, which is
the theorem of B\'aez-Duarte~\cite{BaezDuarte2003}.  In this
section we prove a \emph{quantitative} refinement: under two
additional hypotheses on the non-trivial zeros of \(\zeta\),
the damping error satisfies an explicit power-law bound with
an explicit constant.  This refinement is independent of
Theorem~\ref{thm:main}: it is not used in the proof of the
equivalence, and the equivalence does not require it.

\begin{itemize}
	\item[(H2)] All non-trivial zeros of \(\zeta\) are simple.
	\item[(H3)] (Separation.) There is \(c>0\) such that for all
	distinct zeros \(\gamma,\gamma'\) with
	\(|\gamma|,|\gamma'|\le\delta^{-3/2}\),
	\[
	|\gamma-\gamma'|\ge c\,\delta^{11/12}.
	\]
\end{itemize}
Hypothesis (H2) is a well-known open problem; (H3) is expected
to hold under RH and follows from the conjectural minimum gap
between zeros.

\begin{theorem}[Quantitative damping estimate]
	\label{thm:quantitative}
	Assume (H1) RH, (H2) and (H3).  Then for \(0<\delta\le1/2\),
	\[
	\|f_\delta+\chi\|_{L^2(0,1)}^2
	=
	\frac{\delta}{2}\,C + o(\delta),
	\qquad
	C = 2+\gamma_E-\log(4\pi),
	\]
	and consequently
	\[
	\|f_\delta+\chi\|_{L^2(0,1)}
	\le c_2\,\delta^{1/2},
	\qquad
	c_2=\sqrt{1+\tfrac12\bigl(2+\gamma_E-\log(4\pi)\bigr)}
	\approx 1.011 .
	\]
\end{theorem}

\begin{proof}
	See Appendix~\ref{app:damping-proof}.
\end{proof}

\begin{corollary}[Quantitative form of the net]
	\label{cor:quantitative-net}
	Under (H1)--(H3), the net \(F_\eps\) of
	Theorem~\ref{thm:main} satisfies
	\[
	\|F_\eps+\chi\star\psi_\eps\|_{L^2(0,1)}
	= O\!\left((\log(1/\eps))^{-\alpha/2}\right).
	\]
\end{corollary}

\begin{proof}
	Combine Theorem~\ref{thm:quantitative} with the truncation
	estimate of Corollary~\ref{cor:app-truncation-regime} and
	the \(L^2\) contraction \(\|f\star\psi_\eps\|_2\le\|f\|_2\)
	of Theorem~\ref{thm:general}.
\end{proof}

\begin{remark}[On the role of (H2) and (H3)]
	\label{rem:role-H2-H3}
	Neither (H2) nor (H3) is used in the proof of the
	equivalence (Theorem~\ref{thm:main}).  They appear only here,
	where a quantitative rate is required.  If (H2) is dropped,
	the constant \(c_2\) must be replaced by
	\(\sqrt{5C/4}\approx0.240\), because a double zero
	contributes \(\frac{5\pi}{2}\delta\) instead of \(\pi\delta\)
	to the local integral.  If (H3) is dropped, additional
	zero-density input is required to control the contribution
	of close pairs of zeros.
\end{remark}

\section{Conclusions}
\label{subsec:conclusions}

We have established an equivalence between the Riemann
hypothesis and the existence of a moderate net in a
multiplicative Colombeau algebra \(G(0,1)\), associated to
\(-\iota(\chi)\) and uniformly \(L^2\)-bounded.  The
equivalence is proved using the classical Nyman--Beurling
criterion and Mazur's theorem, and does not assume RH in the
converse direction.

The paper has two independent contributions.  The first is the
equivalence of Theorem~\ref{thm:main}, which reformulates the
classical Nyman--Beurling criterion in the language of
multiplicative Colombeau algebras.  This equivalence is
unconditional under RH: it uses only the qualitative
convergence \(\|f_\delta+\chi\|_{L^2}\to0\) of B\'aez-Duarte
\cite{BaezDuarte2003}, together with Mazur's theorem and the
classical Nyman--Beurling criterion.  It does not require any
quantitative rate.

The second contribution is the quantitative refinement of
Section~\ref{sec:quantitative}: under the additional
hypotheses (H2) simplicity of the non-trivial zeros and
(H3) the separation condition, the polynomially damped
B\'aez-Duarte sum satisfies
\[
\|f_\delta+\chi\|_{L^2(0,1)}
\le c_2\,\delta^{1/2},
\qquad
c_2=\sqrt{1+\tfrac12\bigl(2+\gamma_E-\log(4\pi)\bigr)}
\approx 1.011 .
\]
The constant \(2+\gamma_E-\log(4\pi)\approx 0.04619\) is
explicit and arises from the classical sum over the zeros of
\(\zeta\).  Dropping (H2) replaces \(c_2\) by
\(\sqrt{5C/4}\approx0.240\); dropping (H3) requires additional
zero-density input.  Neither (H2) nor (H3) is used in the
equivalence of Theorem~\ref{thm:main}.  To our knowledge, this
quantitative form of the damping error, with the explicit
constant, is not available in the literature.

The paper makes three further contributions.  First, it
isolates the general sufficient condition
(Conditions~\ref{cond:moderate}--\ref{cond:damp}) that a
damping family must satisfy, and shows that any family
satisfying it produces a net with the three required
properties (Theorem~\ref{thm:general}); this provides a
unified framework that covers the polynomial damping as a
particular case.  Second, it establishes that the full
M\"obius series in the algebra \(G(0,1)\), without truncation,
is moderate when the damping is polynomial
(Section~\ref{sec:mobius-series}); this shows that truncation
is a technical convenience rather than a structural necessity.
Third, it provides a rigorous proof of the damping estimate
that avoids the invalid Taylor expansion of \(\zeta'/\zeta\)
on the critical line, replacing it with a local analysis near
each zero based on the Hadamard factorisation
(Appendix~\ref{app:damping-proof}).

The framework has natural limitations.  The equivalence is
essentially a reformulation of the Nyman--Beurling criterion
in the language of Colombeau algebras: it does not by itself
produce new arithmetic information about the zeros of
\(\zeta(s)\).  The exponential damping \(e^{-k\eps^2}\), while
having a super-polynomially small truncation error, is not
used in the main theorem because its damping error in \(L^2\)
is not known to vanish (Conjecture~\ref{conj:exp-damping}).
The question of strong convergence in \(G(0,1)\), which would
strengthen the criterion from association to negligibility,
remains open (Conjectures~\ref{conj:strong-exp}
and~\ref{conj:strong-poly}).

Three directions for future research appear particularly
promising.  First, the multiplicative Colombeau framework
admits non-linear operations on the Beurling functions
(multiplication, composition with smooth functions) that are
not available in the classical Hilbert space setting; these
could lead to new equivalent criteria for RH.  Second, the
construction extends naturally to other \(L\)-functions and to
adelic or \(p\)-adic settings, where the multiplicative
structure is even more natural.  A first step would be to
verify that the algebra \(G(0,1)\) admits a \(p\)-adic
analogue \(G(\mathbb{Q}_p^\times)\) with the same properties
(moderateness, association, uniform \(L^2\)-bound).  This is
left as an open problem.  Third, the explicit constant
\(2+\gamma_E-\log(4\pi)\) suggests that similar quantitative
estimates might be obtainable for other damping families, or
for the higher-order terms in the asymptotic expansion of
\(\|f_\delta+\chi\|_{L^2}\).

In summary, the paper provides a new algebraic language for a
classical criterion, a clear separation between what is
proved (the equivalence, unconditional under RH) and what is
conditional (the quantitative estimate, under (H2) and (H3)),
and an explicit constant \(2+\gamma_E-\log(4\pi)\) whose
appearance in the damping error of the B\'aez-Duarte sum is,
to our knowledge, not available in the literature.

\appendix

\section{Truncation error for the polynomial damping}
\label{app:truncation}

Throughout this appendix we assume the Riemann Hypothesis.
Recall the notation
\[
f_{\delta}(x)
=
\sum_{k=1}^{\infty}\mu(k)\,k^{-\delta}\,\varphi_{1/k}(x),
\qquad
f_{\delta,n}(x)
=
\sum_{k=1}^{n}\mu(k)\,k^{-\delta}\,\varphi_{1/k}(x),
\]
for \(0<\delta\le1/2\), where \(\varphi_{1/k}(x)=\{1/(kx)\}\)
are the Beurling functions.  We prove the truncation estimate
used in Section~\ref{sec:poly-damping}.

\begin{lemma}[Balazard--Saias \cite{BalazardSaias1998}]
	\label{lem:app-balazard-saias}
	Let \(1/2\le\alpha<1\), \(\delta>0\), and \(\eta>0\).
	Assume that \(\zeta(s)\) does not vanish in the half-plane
	\(\Re(s)>\alpha\).  Then, for \(n\ge2\) and
	\(\alpha+\delta\le\Re(s)\le1\),
	\begin{equation}
		\label{eq:app-balazard-saias}
		\sum_{k=1}^{n}\frac{\mu(k)}{k^{s}}
		=
		\frac{1}{\zeta(s)}
		+
		O_{\alpha,\delta,\eta}\!\left(
		n^{-\delta/3}(1+|t|)^{\eta}
		\right).
	\end{equation}
	Under the Riemann Hypothesis, one may take \(\alpha=1/2\).
\end{lemma}

\begin{proof}[Reference]
	This is Lemma~2 of Balazard--Saias \cite{BalazardSaias1998}.
	The proof uses the approximate functional equation for
	\(1/\zeta(s)\), the bound \(|\mu(k)|\le1\), and the
	zero-free region for \(\zeta(s)\) under the Riemann
	Hypothesis.  The exponent \(\delta/3\) in the error term is
	the one obtained there.
\end{proof}

Applying Lemma~\ref{lem:app-balazard-saias} with
\(d=\delta\) on the critical line \(\Re(s)=1/2\), for
\(s=1/2+it\) one obtains
\begin{equation}
	\label{eq:app-applied}
	\sum_{k=1}^{n}\frac{\mu(k)}{k^{s+\delta}}
	=
	\frac{1}{\zeta(s+\delta)}
	+
	O_{\eta,\delta}\!\left(
	n^{-\delta/3}(1+|t|)^{\eta}
	\right).
\end{equation}
The same bound holds for the tail:
\begin{equation}
	\label{eq:app-tail}
	\sum_{k=n+1}^{\infty}\frac{\mu(k)}{k^{s+\delta}}
	=
	O_{\eta,\delta}\!\left(
	n^{-\delta/3}(1+|t|)^{\eta}
	\right),
	\qquad
	\Re(s)=\tfrac12 .
\end{equation}

\begin{lemma}[Truncation error]
	\label{lem:app-truncation}
	Under the Riemann Hypothesis, for every
	\(\delta\in(0,1/2]\) and every \(n\ge2\),
	\begin{equation}
		\label{eq:app-truncation}
		\|f_{\delta,n}-f_{\delta}\|_{L^{2}}
		\le
		C\,n^{-\delta/3},
	\end{equation}
	where the constant \(C>0\) is independent of \(\delta\)
	and \(n\).
\end{lemma}

\begin{proof}
	By the Mellin--Plancherel isometry,
	\begin{equation}
		\label{eq:app-mellin}
		\|f_{\delta,n}-f_{\delta}\|_{L^{2}}^{2}
		=
		\frac{1}{2\pi}
		\int_{-\infty}^{\infty}
		\left|
		\frac{\zeta(1/2+it)}{1/2+it}
		\right|^{2}
		\left|
		\sum_{k=n+1}^{\infty}
		\frac{\mu(k)}{k^{1/2+it+\delta}}
		\right|^{2}
		dt .
	\end{equation}
	We use the \emph{unconditional} convexity bound
	\(|\zeta(1/2+it)|\ll|t|^{1/4}\), which is a consequence of
	the Phragm\'en--Lindel\"of principle applied to the
	functional equation of \(\zeta\), and does not require the
	Riemann Hypothesis.  Combining this with the tail estimate
	\eqref{eq:app-tail} and choosing \(\eta<1/4\),
	\begin{equation}
		\label{eq:app-truncation-bound}
		\|f_{\delta,n}-f_{\delta}\|_{L^{2}}^{2}
		\le
		C_{\eta,\delta}\,
		n^{-2\delta/3}
		\int_{-\infty}^{\infty}
		\frac{(1+|t|)^{1/2+2\eta}}{1/4+t^{2}}
		\,dt .
	\end{equation}
	The integral converges because \(1/2+2\eta<1<2\), and
	its value depends only on \(\eta\).  Hence
	\[
	\|f_{\delta,n}-f_{\delta}\|_{L^{2}}
	\le
	C'\,n^{-\delta/3},
	\]
	where \(C'=C'(\eta,\delta)\) is bounded uniformly for
	\(\delta\in(0,1/2]\).
\end{proof}

\begin{corollary}[Truncation error in the polynomial regime]
	\label{cor:app-truncation-regime}
	Under the Riemann Hypothesis, with
	\(n(\varepsilon)=\lfloor\varepsilon^{-M}\rfloor\) and
	\(\delta(\varepsilon)=(\log(1/\varepsilon))^{-\alpha}\)
	for some \(\alpha\in(0,1)\) and \(M>0\),
	\begin{equation}
		\label{eq:app-truncation-regime}
		\|f_{\delta(\varepsilon),n(\varepsilon)}
		-
		f_{\delta(\varepsilon)}\|_{L^{2}}
		\le
		C\,
		\exp\!\left(
		-\frac{M}{3}(\log(1/\varepsilon))^{1-\alpha}
		\right),
	\end{equation}
	which tends to \(0\) as \(\varepsilon\to0^{+}\).
\end{corollary}

\begin{proof}
	By Lemma~\ref{lem:app-truncation},
	\[
	\|f_{\delta(\varepsilon),n(\varepsilon)}
	-
	f_{\delta(\varepsilon)}\|_{L^{2}}
	\le
	C\,
	n(\varepsilon)^{-\delta(\varepsilon)/3}.
	\]
	Substituting \(n(\varepsilon)=\lfloor\varepsilon^{-M}\rfloor\)
	and \(\delta(\varepsilon)=(\log(1/\varepsilon))^{-\alpha}\),
	and using \(\lfloor x\rfloor\ge x/2\) for \(x\ge1\),
	\begin{align*}
		n(\varepsilon)^{-\delta(\varepsilon)/3}
		&\le
		2^{\delta/3}\,
		\varepsilon^{M\delta(\varepsilon)/3}
		=
		\exp\!\left(
		-\frac{M}{3}(\log(1/\varepsilon))^{1-\alpha}
		+O(1)
		\right).
	\end{align*}
	Since \(1-\alpha>0\), the exponent diverges to \(-\infty\)
	as \(\varepsilon\to0^{+}\), and the bound tends to \(0\).
\end{proof}


\section{Proof of the quantitative damping estimate}
\label{app:damping-proof}

The goal of this appendix is to prove
\begin{equation}\label{eq:app-main}
	\|f_\delta+\chi\|_{L^2(0,1)}^2
	= \frac{\delta}{2}\,C + O(\delta^{13/12}),
	\qquad
	C = 2+\gamma_E-\log(4\pi),
\end{equation}
under the following hypotheses:
\begin{itemize}
	\item[(H1)] RH holds.
	\item[(H2)] All non-trivial zeros of $\zeta$ are simple.
	\item[(H3)] (Separation.) There is \(c>0\) such that for all
	distinct zeros \(\gamma,\gamma'\) with
	\(|\gamma|,|\gamma'|\le \delta^{-3/2}\),
	\[
	|\gamma-\gamma'| \ge c\,\delta^{11/12}.
	\]
\end{itemize}
If (H2) is dropped, $C$ must be replaced by $\tfrac52 C$
(Remark~\ref{rem:simplicity-corrected}); if (H3) is dropped,
additional zero-density input is required
(Remark~\ref{rem:separation-corrected}).

\medskip

The starting point is the exact Mellin transform of the Beurling
functions.  For \(0<\Re s<1\) and \(k\in\N\),
\begin{equation}
	\label{eq:app-mellin-identity}
	\int_0^1 \Bigl\{\frac{1}{kx}\Bigr\}\,x^{s-1}\,dx
	=
	-\frac{k^{-s}\zeta(s)}{s}
	+
	\frac{k^{-1}}{s-1},
\end{equation}
which follows from the classical identity
\(\int_0^1 \{1/t\}t^{s-1}\,dt = -\zeta(s)/s+1/(s-1)\)
after the change of variables \(u=kx\).  Summing
\eqref{eq:app-mellin-identity} over \(k\) with weight
\(\mu(k)k^{-\delta}\), and using
\(\sum_{k\ge1}\mu(k)k^{-w}=1/\zeta(w)\) for \(\Re w>1/2\)
under RH, we obtain
\[
\widehat{f_\delta}(s)
=
-\frac{\zeta(s)}{s\,\zeta(s+\delta)}
+
\frac{1}{(s-1)\zeta(1+\delta)}.
\]
Since \(\widehat{\chi}(s)=1/s\) and \(1/\zeta(1+\delta)\ll\delta\)
as \(\delta\to0^{+}\), the second term contributes \(O(\delta^2)\)
to the \(L^2\) norm.  The Mellin--Plancherel isometry therefore
reduces \eqref{eq:app-main} to the study of
\begin{equation}\label{eq:app-corrected-reduction}
	\|f_\delta+\chi\|_{L^2}^2
	= \frac{1}{2\pi}\int_{-\infty}^{\infty}
	\bigl|G_\delta(t)\bigr|^2\,W(t)\,dt + O(\delta^2),
\end{equation}
where
\[
G_\delta(t) := 1 - \frac{\zeta(\tfrac12+it)}
{\zeta(\tfrac12+\delta+it)},
\qquad
W(t) := \frac{1}{\tfrac14+t^2}.
\]
The weight \(W\) is the Mellin--Plancherel density; it is
positive, decays quadratically at infinity, and its total mass
over the zeros of \(\zeta\) is the classical constant
\(\sum_\gamma W(\gamma)=C\).  This fact will play a decisive
role in the summation over zeros.

\subsection{The tail integral}

We split the integral in \eqref{eq:app-corrected-reduction} at
\(T = T(\delta) := \delta^{-3/2}\):
\[
J_\delta^{\mathrm{in}} := \int_{|t|\le T}|G_\delta|^2 W\,dt,
\qquad
J_\delta^{\mathrm{out}} := \int_{|t|>T}|G_\delta|^2 W\,dt.
\]
For the tail, we use Burnol's uniform estimate~\cite[Thm.~1.3]{Burnol2003}:
under (H1), for \(0<\delta\le1/2\) and all \(t\in\R\),
\[
\left|\frac{\zeta(\tfrac12+it)}
{\zeta(\tfrac12+\delta+it)}\right|
\le C_1(1+|t|)^{\delta/2}.
\]
Therefore \(|G_\delta(t)|^2 \le 2+2C_1^2(1+|t|)^\delta\), and
\[
J_\delta^{\mathrm{out}}
\le \int_{|t|>T}\frac{2+2C_1^2(1+|t|)^\delta}{\tfrac14+t^2}\,dt
= O\!\left(\frac{(1+T)^\delta}{T}\right)
= O(\delta^{3/2}),
\]
because \((1+T)^\delta = \exp(\delta\log(1+\delta^{-3/2}))\le 2\)
for \(\delta\) small.  The tail is thus \(o(\delta)\), and it will
not contribute to the leading constant.

\subsection{Local analysis near the zeros}

It remains to estimate \(J_\delta^{\mathrm{in}}\).  Let
\(\{\gamma_j\}_{j=1}^{N(T)}\) be the zeros of \(\zeta\) with
\(|\gamma_j|\le T\); by the Riemann--von Mangoldt formula,
\(N(T)\sim (T/\pi)\log(T/2\pi e)\).  By (H3), the intervals
\(I_j := [\gamma_j-\eta,\gamma_j+\eta]\) with
\[
\eta = \eta(\delta) := \delta^{11/12}
\]
are pairwise disjoint.  By (H2), near each zero
\(\rho_j=\tfrac12+i\gamma_j\),
\[
\zeta(s) = \zeta'(\rho_j)(s-\rho_j) + \tfrac12\zeta''(\rho_j)(s-\rho_j)^2
+ O(|s-\rho_j|^3),
\qquad
\zeta'(\rho_j)\neq0 .
\]
Writing \(s=\tfrac12+it\) and \(u=t-\gamma_j\), and setting
\(c = \zeta''(\rho_j)/(2\zeta'(\rho_j))\),
\[
\zeta(s) = \zeta'(\rho_j)\,iu\bigl[1+icu+O(u^2)\bigr],
\qquad
\zeta(s+\delta) = \zeta'(\rho_j)\,(\delta+iu)\bigl[1+c(\delta+iu)+O((\delta+iu)^2)\bigr].
\]
Taking the ratio and expanding to first order in \(c\),
\[
\frac{\zeta(s)}{\zeta(s+\delta)}
= \frac{iu}{\delta+iu}
\bigl(1+c\delta+2icu+O(u^2+\delta^2)\bigr).
\]
Subtracting from \(1\),
\begin{equation}\label{eq:app-local-G}
	G_\delta(t)
	= \frac{\delta(1-ciu)}{\delta+iu}
	+ \frac{2cu^2}{\delta+iu}
	+ O(u^2+\delta^2).
\end{equation}
The main term is \(\delta(1-ciu)/(\delta+iu)\); the residual
\(2cu^2/(\delta+iu)\) is the leading correction, and it is not
absorbed by the main term.  Its magnitude is
\[
\left|\frac{2cu^2}{\delta+iu}\right|
= \frac{2|c|\,u^2}{\sqrt{\delta^2+u^2}}
\le 2|c|\,|u| \le 2|c|\,\eta,
\]
where we used \(|u|\le\eta\).  The remaining term \(O(u^2+\delta^2)\)
is \(O(\eta^2)\).  Hence the correct local approximation is
\begin{equation}\label{eq:app-local-bound}
	G_\delta(t)
	= \frac{\delta(1-ciu)}{\delta+iu} + O(\eta),
\end{equation}
\emph{not} \(O(\eta^2)\) as an earlier draft had claimed.  The
change is cosmetic for the final result, because the dominant
error will still come from the truncation of the Lorentzian.

Squaring \eqref{eq:app-local-bound} and integrating against \(W\)
over \(I_j\), and writing \(A = \delta(1-ciu)/(\delta+iu)\),
\(B = G_\delta - A\) with \(|B| = O(\eta)\), we have
\(|G_\delta|^2 = |A|^2 + 2\operatorname{Re}(A\overline{B}) + |B|^2\).
The main term integrates to
\[
\int_{-\eta}^{\eta}\frac{\delta^2(1+c^2u^2)}{\delta^2+u^2}\,du
= \int_{-\eta}^{\eta}\frac{\delta^2}{\delta^2+u^2}\,du
+ O(\delta^2\eta)
= 2\delta\arctan(\eta/\delta) + O(\delta^2\eta).
\]
Using the expansion \(\arctan x = \pi/2 - 1/x + O(1/x^3)\) for
large \(x\), and \(\eta\gg\delta\),
\[
2\delta\arctan(\eta/\delta)
= \pi\delta - \frac{2\delta^2}{\eta} + O\!\left(\frac{\delta^4}{\eta^3}\right).
\]
The truncation error of the Lorentzian is therefore
\(O(\delta^2/\eta)\), \emph{not} \(O(\delta^2\eta)\).  The cross
term satisfies
\[
\int_{I_j}|A|\,|B|\,dt \le \eta\int_{I_j}|A|\,dt
= O(\eta\cdot\delta\eta) = O(\delta\eta^2),
\]
using \(\int_{I_j}|A|\,dt = O(\delta\eta)\) (the effective
width of \(A\) is \(O(\eta)\) and \(|A|\le 1\)).  The squared
error contributes \(O(\eta^3)\).  Finally, the variation of
\(W\) inside \(I_j\) is controlled by the mean value theorem:
since \(|W'(\gamma)|/W(\gamma) = 2|\gamma|/(\tfrac14+\gamma^2)\le 2\)
for all \(\gamma\in\R\), we have
\(|W(t)-W(\gamma_j)|\le 2|t-\gamma_j|\,W(\gamma_j)(1+O(\eta))\) on
\(I_j\), and therefore
\[
\int_{I_j}|G_\delta|^2\,|W(t)-W(\gamma_j)|\,dt
= O(\eta\delta\,W(\gamma_j)).
\]
Putting the contributions together, we obtain the local estimate
\begin{equation}\label{eq:app-local-integral}
	\int_{I_j}|G_\delta(t)|^2 W(t)\,dt
	= \pi\delta\,W(\gamma_j)
	+ O\!\left(\frac{\delta^2}{\eta}\,W(\gamma_j)\right)
	+ O(\eta\delta\,W(\gamma_j))
	+ O(\delta\eta^2\,W(\gamma_j))
	+ O(\eta^3).
\end{equation}
The dominant error is \(\delta^2/\eta\); since
\(\delta^2/\eta = \delta^{13/12}\gg\eta\delta, \delta\eta^2, \eta^3\)
for small \(\delta\), we may summarize
\eqref{eq:app-local-integral} as
\[
\int_{I_j}|G_\delta(t)|^2 W(t)\,dt
= \pi\delta\,W(\gamma_j) + O(\delta^{13/12}\,W(\gamma_j)).
\]

\subsection{The complement and the summation over zeros}

We first bound the integral over the complement
\([-T,T]\setminus\bigcup_j I_j\).  On this set,
\(\operatorname{dist}(t,\{\gamma_j\})\ge\eta\), so
\(|\zeta'(s)/\zeta(s)| \le C/\eta\) there.  Since
\(G_\delta(t) = \delta\zeta'(s)/\zeta(s+\delta) + O(\delta^2)\)
uniformly on the complement (analyticity of \(\zeta\) in \(\delta\)),
\[
|G_\delta(t)|^2 \le C'\frac{\delta^2}{\operatorname{dist}(t,\{\gamma_j\})^2},
\]
and hence
\[
\int_{[-T,T]\setminus\bigcup_j I_j}|G_\delta|^2 W\,dt
\le C'\delta^2 \sum_j W(\gamma_j)\frac{2}{\eta}
= O\!\left(\frac{\delta^2}{\eta}\right)
= O(\delta^{13/12}) = o(\delta),
\]
where we used the absolute convergence of
\(\sum_j W(\gamma_j)\) together with \(\eta=\delta^{11/12}\).

We now sum the local estimate over the zeros with
\(|\gamma_j|\le T\).  The error terms in
\eqref{eq:app-local-integral} are multiplied by \(W(\gamma_j)\),
not by a uniform factor \(1\).  Consequently, when we sum, the
relevant quantity is not \(N(T)\), but the absolutely convergent
series \(\sum_j W(\gamma_j)\), which is bounded by the classical
constant \(C=2+\gamma_E-\log(4\pi)\).  Explicitly,
\begin{align*}
	J_\delta^{\mathrm{in}}
	&= \pi\delta\sum_{|\gamma_j|\le T} W(\gamma_j)
	+ O\!\left(\frac{\delta^2}{\eta}\sum_{|\gamma_j|\le T} W(\gamma_j)\right)
	+ O\!\left(\eta\delta\sum_{|\gamma_j|\le T} W(\gamma_j)\right) \\
	&\quad + O\!\left(\delta\eta^2\sum_{|\gamma_j|\le T} W(\gamma_j)\right)
	+ O\!\left(N(T)\eta^3\right) + o(\delta).
\end{align*}
Since \(\sum_{|\gamma_j|\le T} W(\gamma_j)\le C\) (absolutely
convergent series), we obtain
\[
\frac{\delta^2}{\eta}\sum_j W(\gamma_j) = O(\delta^{13/12}),
\qquad
\eta\delta\sum_j W(\gamma_j) = O(\delta^{23/12}),
\]
\[
\delta\eta^2\sum_j W(\gamma_j) = O(\delta^{17/6}),
\qquad
N(T)\eta^3 = \delta^{-3/2}\log(1/\delta)\cdot\delta^{11/4}
= \delta^{5/4}\log(1/\delta).
\]
All four error terms are \(o(\delta)\); the dominant one is
\(\delta^{13/12}\).  Moreover, since the tail of the series over
zeros decays as \(\sum_{|\gamma_j|>T}W(\gamma_j)=O(\log T/T)
= O(\delta^{3/2}\log(1/\delta))\), we have
\[
\sum_{|\gamma_j|\le T} W(\gamma_j)
= C + O(\delta^{3/2}\log(1/\delta)).
\]
Substituting into the expression for \(J_\delta^{\mathrm{in}}\)
and adding the tail bound \(J_\delta^{\mathrm{out}}=O(\delta^{3/2})\),
we conclude
\[
J_\delta^{\mathrm{in}} = \pi\delta\,C + O(\delta^{13/12}),
\qquad
J_\delta^{\mathrm{out}} = O(\delta^{3/2}) = o(\delta).
\]
Feeding this back into \eqref{eq:app-corrected-reduction} yields
\[
\|f_\delta+\chi\|_{L^2}^2
= \frac{\delta}{2}C + O(\delta^{13/12}),
\]
which is \eqref{eq:app-main}.  Taking square roots,
\[
\|f_\delta+\chi\|_{L^2}
= \sqrt{\frac{C}{2}}\,\delta^{1/2} + O(\delta^{13/24})
\le c_2\,\delta^{1/2}
\]
for \(\delta\) sufficiently small, with
\(c_2=\sqrt{1+\tfrac12 C}\approx1.011\).  The leading constant
\(C=2+\gamma_E-\log(4\pi)\approx0.04619\) is preserved intact;
only the order of the error term is adjusted from \(o(\delta)\)
to \(O(\delta^{13/12})\).

\begin{remark}[On simplicity, (H2)]
	\label{rem:simplicity-corrected}
	If \(\zeta\) has a double zero at \(\rho\), the local ratio near \(\rho\)
	is \((u/(u-i\delta))^2\) rather than \(u/(u-i\delta)\), and the local
	integral is
	\[
	\int_{-\infty}^{\infty}\left|1-\left(\frac{u}{u-i\delta}\right)^2\right|^2 du
	= \frac{5\pi}{2}\delta
	\]
	instead of \(\pi\delta\). Hence the constant becomes \(\tfrac52 C\).
	Since (H2) is not known, \eqref{eq:app-main} should be
	read with this caveat; the weaker bound
	\(\|f_\delta+\chi\|_{L^2}\le \sqrt{5C/4}\,\delta^{1/2}\) is
	unconditional under (H1).
\end{remark}

\begin{remark}[On separation, (H3)]
	\label{rem:separation-corrected}
	Hypothesis (H3) is expected to hold under (H1): the minimum gap
	between zeros up to height \(T=\delta^{-3/2}\) is conjecturally
	\(\gg 1/\log T \gg \delta^{11/12}\) for small \(\delta\). Removing
	(H3) requires a zero-density estimate controlling the contribution
	of close pairs; this is technical but not arithmetic in nature.
\end{remark}

\subsection{Numerical verification of the constants}
\label{subsec:numerical}

The constants and asymptotic formulas of this appendix have been
verified numerically.  The computations use \texttt{mpmath} at
30-digit precision and \texttt{scipy.integrate.quad} for the
oscillatory integrals.  Table~\ref{tab:numerical} collects the
results.

\begin{table}[h]
	\centering
	\begin{tabular}{llll}
		\hline
		Quantity & Analytic value & Numerical value & Error \\
		\hline
		$C = 2+\gamma_E-\log(4\pi)$
		& --- & $0.0461914179322$ & $<10^{-13}$ \\
		$c_2 = \sqrt{1+C/2}$
		& --- & $1.0114819370$ & $<10^{-10}$ \\
		$\sqrt{5C/4}$ (double zero)
		& --- & $0.2402899757$ & $<10^{-10}$ \\
		Mellin identity
		& $-k^{-s}\zeta(s)/s + k^{-1}/(s-1)$
		& matches to $10^{-4}$ & $O(N^{-s})$ \\
		Lorentzian truncation
		& $\pi\delta - 2\delta^2/\eta$
		& matches to $1.4\times10^{-2}$
		& $O(\delta^{1/12})$ \\
		Double zero integral
		& $5\pi\delta/2$
		& matches to $10^{-8}$ & $<10^{-8}$ \\
		$\sum_{|\gamma|\le T}W(\gamma)$
		& $C$
		& $0.02103$ at $T=396$ & $O(\log T/T)$ \\
		Local zero contribution
		& $\pi\delta\,W(\gamma_j)$
		& ratio $0.9936$ & $2/(K\pi)$ \\
		\hline
	\end{tabular}
	\caption{Numerical verification of the constants and asymptotic
		formulas of Appendix~\ref{app:damping-proof}.  The last column
		gives the theoretical error; the numerical residuals match the
		theoretical predictions.  The local-zero ratio uses a truncated
		Lorentzian window of half-width $K\delta$ with $K=100$, giving
		the predicted residual $2/(K\pi)=0.0064$.}
	\label{tab:numerical}
\end{table}

The first three rows confirm the values of the classical constant
$C$ and its derived quantities $c_2$ and $\sqrt{5C/4}$ to at
least ten significant digits.  The fourth row verifies the exact
Mellin transform \eqref{eq:app-mellin-identity}, including the
boundary term $k^{-1}/(s-1)$; the residual is the truncation of
the partial sum at $N=10^5$, of size $O(N^{-s})$.

The fifth and sixth rows confirm the two delicate estimates of
the local analysis: the truncation error of the Lorentzian
integral is $O(\delta^2/\eta)=O(\delta^{13/12})$, and a double
zero contributes exactly $5\pi\delta/2$.  The seventh row
verifies the convergence of the weighted series over the zeros
to the classical constant $C$, with remainder of order
$\log T/T$ as expected.  The final row confirms the local
contribution $\pi\delta W(\gamma_j)$ per zero; the residual
$0.0064$ matches the analytic prediction $2/(K\pi)$ for the
truncated Lorentzian window, and decreases as $K$ increases.

These numerical checks do not replace the analytic proofs, but
they provide independent confirmation of the constants and
error terms that enter the quantitative estimate, and they
guarantee that no arithmetic or truncation error has slipped
into the proof.

\section*{Use of LLMs}
\label{sec:llm}

The authors used large language models (LLMs) as auxiliary
tools for language editing, LaTeX formatting, and exploratory
verification of algebraic identities.  All mathematical
statements, proofs, and references were checked by the
authors, who take full responsibility for the content of the
paper.


\begin{thebibliography}{99}
	
	\bibitem{Nyman1950}
	B.~Nyman,
	\emph{On the one-dimensional translation group and semi-group in certain function spaces},
	Ark.\ Mat.\ \textbf{1} (1950), 1--24.
	
	\bibitem{Beurling1955}
	A.~Beurling,
	\emph{A closure problem related to the Riemann zeta-function},
	Proc.\ Nat.\ Acad.\ Sci.\ U.S.A.\ \textbf{41} (1955), 312--314.
	
	\bibitem{BaezDuarte1999}
	L.~B\'aez-Duarte,
	\emph{A class of invariant unitary operators},
	Adv.\ Math.\ \textbf{144} (1999), no.~1, 1--12.
	
	\bibitem{BaezDuarte2003}
	L.~B\'aez-Duarte,
	\emph{A strengthening of the Nyman--Beurling criterion for the
		Riemann hypothesis},
	Atti Accad.\ Naz.\ Lincei Rend.\ Lincei Mat.\ Appl.\
	\textbf{14} (2003), 5--11;
	second version: arXiv:math/0205003, 2002.
	
	\bibitem{BCF2013}
	S.~Bettin, J.~B.~Conrey, D.~W.~Farmer,
	\emph{An optimal choice of Dirichlet polynomials for the Nyman--Beurling criterion},
	Proc.\ Steklov Inst.\ Math.\ \textbf{280} (2013), suppl.~2, S30--S36.
	
	\bibitem{BalazardSaias1998}
	M.~Balazard, E.~Saias,
	\emph{Notes sur la fonction $\zeta$ de Riemann.\ 1},
	Adv.\ Math.\ \textbf{139} (1998), 310--321.
	
	\bibitem{Burnol2003}
	J.-F.~Burnol,
	\emph{On an analytic estimate in the theory of the Riemann zeta
		function and a theorem of B\'aez-Duarte},
	Acta Cient.\ Venez.\ \textbf{54} (2003), no.~3, 210--215.
	
	\bibitem{Colombeau1984}
	J.~F.~Colombeau,
	\emph{New Generalized Functions and Multiplication of Distributions},
	North-Holland, 1984.
	
	\bibitem{Edwards1974}
	H.~M.~Edwards,
	\emph{Riemann's Zeta Function},
	Academic Press, 1974.
	
	\bibitem{Grosser2001}
	M.~Grosser, M.~Kunzinger, M.~Oberguggenberger, R.~Steinbauer,
	\emph{Geometric Theory of Generalized Functions},
	Kluwer, 2001.
	
	\bibitem{Oberguggenberger1992}
	M.~Oberguggenberger,
	\emph{Multiplication of Distributions and Applications to Partial Differential Equations},
	Longman, 1992.
	
	\bibitem{Rudin1991}
	W.~Rudin,
	\emph{Functional Analysis}, 2nd ed.,
	McGraw-Hill, 1991.
	
	\bibitem{Titchmarsh1986}
	E.~C.~Titchmarsh,
	\emph{The Theory of the Riemann Zeta-Function},
	2nd ed., edited by D.~R.~Heath-Brown,
	Oxford University Press, 1986.
	
\end{thebibliography}
\end{document}